\documentclass[11pt]{amsart}
  \usepackage[margin=3.3cm]{geometry}

\usepackage{amsfonts,amsmath,amssymb,color,esint}
\usepackage{enumerate}
\usepackage{graphicx}
\usepackage{tikz}
\usetikzlibrary{decorations.markings}
\usetikzlibrary{plotmarks}
\usetikzlibrary{patterns}
\numberwithin{equation}{section}

\usepackage{amsthm}

\usepackage[babel=true,kerning=true]{microtype}

\usepackage{amsthm,leqno}
\usepackage{hyperref}

\newtheorem{conj}{Conjecture}[section]
\newtheorem{remark}{Remark}[section]
\newtheorem{theo}{Theorem}[section]
\newtheorem{lem}{Lemma}[section]

\newtheorem{cl}{Claim}[section]
\newtheorem{prop}{Proposition}[section]

\newcommand{\eps}{\varepsilon}

\newcommand{\R}{\mathbb{R}}

\begin{document}

\title[]{Isoperimetric inequalities involving Steklov eigenvalues on surfaces}
\author{Romain Petrides}
\makeatletter
\let\@wraptoccontribs\wraptoccontribs
\makeatother
\contrib[with an appendix by]{Henrik Matthiesen and Romain Petrides}
\address{Romain Petrides, Universit\'e de Paris, Institut de Math\'ematiques de Jussieu - Paris Rive Gauche, b\^atiment Sophie Germain, 75205 PARIS Cedex 13, France}
\email{romain.petrides@imj-prg.fr}

\begin{abstract} 
We give results on optimal constants of isoperimetric inequalities involving Steklov eigenvalues on surfaces with boundary. We both consider this question on Riemannian surfaces with a same given topology or more specifically belonging to the same conformal class. We provide new examples of topological disks that realize optimal constants. We prove inequalities that relate conformal invariants associated to combinations of Steklov eigenvalues on a compact Riemannian surface with boundary and the ones on the disk. In the appendix, we show rigidity of the first conformal Steklov eigenvalue on annuli and M{\"o}bius bands.
\end{abstract}

\maketitle

Let $(\Sigma,g)$ be a compact Riemannian surface with a non empty boundary. Let $\beta$ be a positive function in $\partial \Sigma$. We set
$$ \sigma_k(\Sigma,g,\beta) = \inf_{V \in \mathcal{G}_{k+1}(\mathcal{C}^\infty(\Sigma))} \max_{\phi \in V \setminus \{0\}}  \frac{\int_\Sigma \vert \nabla \phi \vert^2_g dA_g}{\int_{\partial \Sigma} \phi^2 \beta dL_g},$$
where $dA_g$ is the area measure of $\Sigma$ with respect to $g$, $dL_g$ is the length measure of $\partial \Sigma$ with respect to $g$ and $\mathcal{G}_{k+1}(\mathcal{C}^\infty(\Sigma))$ is the set of $k+1$-dimensional subspaces of $\mathcal{C}^\infty(\Sigma)$.
We obtain a sequence of so-called (weighted) Steklov eigenvalues
$$ 0 = \sigma_0 \leq \sigma_1(\Sigma,g,\beta) \leq \cdots \leq \sigma_k(\Sigma,g,\beta) \leq \cdots \to +\infty \text{ as } k \to +\infty, $$
where $\sigma_0$ is associated to constant functions of $\Sigma$ and $\sigma_1(\Sigma,g,\beta) >0$ if $\Sigma$ is a connected surface. By conformal invariance of the Dirichlet energy, we have for any smooth extension $\hat{\beta}$ on $\Sigma$ of $\beta$,
\begin{equation} \label{conformalcovariance}\sigma_k(\Sigma,g,\beta)  = \sigma_k(\Sigma,\hat{g}, 1), \end{equation}
where $\hat{g} = \hat{\beta}^2 g$. Therefore, we say that $(\Sigma,g,\beta)$ is Steklov-isometric to $(\tilde{\Sigma},\tilde{g})$ if there is a diffeomorphism $\theta : (\Sigma,\partial \Sigma) \to (\tilde{\Sigma},\partial \tilde{\Sigma}) $ such that $\theta^\star(\tilde{g}) = \hat{\beta}^2 g$ for any smooth extension $\hat{\beta}$ on $\Sigma$ of $\beta$. More generally, $(\Sigma,g,\beta)$ and $(\tilde{\Sigma},\tilde{g},\tilde{\beta})$ are Steklov isometric if there is a diffeomorphism $\theta : (\Sigma,\partial \Sigma) \to (\tilde{\Sigma},\partial \tilde{\Sigma}) $ such that $\theta^\star( \hat{\tilde{\beta}}^2 \tilde{g} ) = \hat{\beta}^2 g$,for any smooth extension $\hat{\beta}$ on $\Sigma$ of $\beta$ and $\hat{\tilde{\beta}}$ on $\tilde{\Sigma}$ of $\tilde{\beta}$.

We refer to the surveys \cite{CGGS,gp} for recent developments on Steklov eigenvalues. In the current paper, we would like to set families of isoperimetric inequalities involving Steklov eigenvalues of $\Sigma$, to give optimal constants and to look for metrics that realize these sharp inequalities. We then set the following scaling invariant functional
$$ \bar{\sigma}_k(\Sigma,g,\beta) = \sigma_k(\Sigma,g,\beta) \int_{\partial \Sigma} \beta dL_g $$
so-called the renormalized eigenvalue by the perimeter of the boundary of $(\Sigma,\tilde{g})$, that is the length of $\partial \Sigma$ with respect to $\tilde{g} = \beta^2 g$. We denote for $k\geq 1$:
\begin{equation} \label{def:sigmakconf} \sigma_k(\Sigma,[g]) = \sup_{\beta \in \mathcal{C}^\infty_{>0}(\Sigma)} \bar{\sigma}_k(\Sigma,g,\beta), \end{equation}
that depends only on the conformal class of $g$ by \eqref{conformalcovariance} and
\begin{equation} \label{def:sigmak} \sigma_k(\Sigma) = \sup_{g\in Met(\Sigma),\beta \in \mathcal{C}^\infty_{>0}(\Sigma)} \bar{\sigma}_k(\Sigma,g,\beta) = \sup_{g \in Met(\Sigma)} \sigma_k(\Sigma,[g]), \end{equation}
where $Met(\Sigma)$ denotes the set of smooth riemannian metrics of $\Sigma$. A result by Hassannezhad \cite{GirouardPolterovich,hassannezhad} states that $\sigma_k(\Sigma,[g]) <+\infty$ and that  $\sigma_k(\Sigma) < +\infty$. These invariants are well-known on topological disks by \cite{HerschPayneSchiffer}:
$$ \sigma_k(\mathbb{D}) = 2\pi k. $$
Notice that by the uniformization theorem for any $g \in Met(\Sigma)$, $\sigma_k(\mathbb{D}) = \sigma_k(\mathbb{D},[g])$. For $k=1$, it was already computed in \cite{weinstock} and it is realized if and only if $(\Sigma,g,\beta)$ is Steklov-isometric to an Euclidean disk. However, using the main result in \cite{FraserSchoendisk}, and Theorem \ref{theo:critcombS}, Euclidean disks are the only critical points of $\bar{\sigma}_k$ and $\sigma_k(\mathbb{D}) = 2\pi k$ is never realized for $k\geq 2$. The value $\sigma_k(\mathbb{D}) = 2\pi k$ is the $k$-th eigenvalue of the disjoint union of $k$ isometric Euclidean disks. Actually, we only know that the following surfaces of higher topologies have maximal metrics for $\bar{\sigma}_1$: the annulus and Mobius band \cite{fs} and orientable surfaces of genus $0$ and $1$ \cite{KKMS}. In addition, we have for $\gamma =0$ by \cite{GL}
$$ \lim_{b \to +\infty} \sigma_1(\Sigma) = 8\pi,$$
where $8\pi$ corresponds to the maximal value of the renormalized first Laplace eigenvalue on spheres. This result is extended to any genus of $\Sigma$ \cite{KSduke}.

Despite these recent advances, computations with respect to $k$ and the topology of $\Sigma$ of sharp constants $\sigma_k(\Sigma)$ and associated maximizers if they exist are still widely open.

\medskip

Considering combinations of eigenvalues improves the general picture. Let $f : \R^m_+ \to \mathbb{R} \cup \{+\infty\}$ be a $\mathcal{C}^1$ function that satisfies $\partial_i f(x) \leq 0$ for all $x \in \R^m_+$. We set
$$E_{f}^S(\Sigma,g,\beta) = f(\bar{\sigma}_1(\Sigma,g,\beta),\cdots,\bar{\sigma}_m(\Sigma,g,\beta))$$
a Steklov spectral functional. We aim at getting sharp lower bounds on $E_f^S$ with respect to the topology and geometry of $(\Sigma,g)$, and at describing the minimizers if they exist. We set
$$ I^S(\Sigma,f) = \inf_{g \in Met(\Sigma),\beta \in  \mathcal{C}^\infty_{>0}(\Sigma) } E_{f}^S(\Sigma,g,\beta) = \inf_{g \in Met(\Sigma)} I^S_c(\Sigma,[g],f) , $$
and
$$ I^S_c(\Sigma,[g],f) = \inf_{\beta \in \mathcal{C}^\infty_{>0}(\Sigma)} E_{f}^S(\Sigma,g,\beta). $$
Notice again that $I^S(\mathbb{D},f) = I^S_c(\mathbb{D},[g],f) $ for any metric in $\mathbb{D}$.

Let's give examples of Steklov spectral functionals that where previously studied:
\begin{itemize}
\item Minimization of generalized Hersch-Payne functionals \cite{HerschPayne,Petridesnonplanardisk} for $t \geq 0$:
$$ h_t^+(x) =  \frac{1}{x_1} + \frac{t}{x_2}.$$
For $t \leq 1$, every minimizer for $I^S(\mathbb{D},h_t)$ is Steklov-isometric to the Euclidean disk and we have by \cite{HerschPayne}
\begin{equation} \label{eq:herschpayne} I^S(\mathbb{D},h_t^+) = \frac{1+t}{2\pi}. \end{equation}
It generalized Weinstock's result ($t=0$) \cite{weinstock} to $0 < t \leq 1$. In \cite{Petridesnonplanardisk}, we prove that for $t>1$, this functional is attained, but never attained by the Euclidean disk, and also that there is a minimal positive number $t_0$ such that for all $t > t_0$, it is attained by a surface associated to a non planar free boundary minimal disk into an ellipsoid of $\R^3$. In Proposition \ref{prop:nonplanarbifurcation} below, we prove that $t_0 \leq 3$.
As formulated in Conjecture \ref{conj} below, we expect to have $t_0 = 3$.
\item Maximization of product of eigenvalues or Hersch-Payne-Schiffer functionals \cite{HerschPayneSchiffer} for $m,n \geq 1$:
$$ f_{m,n}(x) = \left(x_m x_n\right)^{-1}.$$
By \cite{GirouardPolterovich}, we have on an orientable connected surface of genus $\gamma$ with $b$ boundary components that
$$ \left( I^S(\Sigma,f_{m,n}) \right)^{-1} \leq \begin{cases} \pi^2 (\gamma+b)^2(m+n)^2 \text{ if } m+n \text{ is even,} \\
\pi^2 (\gamma+b)^2(m+n-1)^2 \text{ if } m+n \text{ is odd.}
\end{cases} $$
If $m=n$, $\gamma = 0$ and $b=1$, this is an equality by \cite{HerschPayneSchiffer}. Moreover, if $m = n+1$, $\gamma = 0$ and $b=1$, this is also an equality and it is attained if and only if $m=1$ and $n=2$. In this case, the Euclidean disk is again a minimizer of $I^S(\mathbb{D},f_{1,2})$. In the current paper, we prove that the set of minimizers of $I^S(\mathbb{D},f_{1,2})$ is a one parameter infinite set (see Proposition \ref{prop:crithps})
\item Maximization of linear combinations of the two first non zero eigenvalues \cite{Petridesnonplanardisk}:
$$ h_t^-(x) = \left(x_1 + t x_2\right)^{-1} $$
In \cite{Petridesnonplanardisk}, we proved that the minimum is always realized, and that for $t$ large enough, it is attained by a surface associated to a non planar free boundary minimal disk into an ellipsoid of $\R^3$.
\end{itemize}

Of course, such studies could be asked for general linear combinations of eigenvalues or of inverse of eigenvalues and other negative combinations of eigenvalues. One important step to study these problems is the explicit computation of critical points of these functionals.
\begin{theo}[\cite{PetridesCS,PetridesTewodrose}] \label{theo:critcombS}
Let $(\Sigma,g)$ be a compact Riemannian surface with boundary. We assume that $(g,\beta)$ is a critical point for $E_f^S(\Sigma,\cdot,\cdot)$. Then there is a map $\Phi : \Sigma \to \R^n$ such that
\begin{itemize}
\item[(1)] For $1 \leq i \leq n$, the coordinate $\phi_i$ is an eigenfunction associated to  $\sigma_i := \sigma_i(\Sigma,g,\beta)$. Denoting $\sigma = diag(\sigma_1,\cdots,\sigma_{m-1},\sigma_m,\cdots, \sigma_m)$, we write the equation:
$$ \begin{cases}\Delta_g \Phi = 0 \text{ in } \Sigma \\
\partial_\nu \Phi = \beta \sigma \cdot \Phi \text{ on } \partial \Sigma \end{cases} $$
\item[(2)] $\vert \Phi \vert_{\sigma} = 1$ on $\partial \Sigma$
\item[(3)] $d\Phi \otimes d\Phi = \frac{\vert \nabla \Phi \vert_g^2}{2} g$
\item[(4)] For all $1 \leq i \leq m$, denoting $I_i = \{ 1\leq j \leq n ; \sigma_j = \sigma_i \}$, we have
$$ \sum_{j \in I_i} \int_{\partial\Sigma} \phi_i^2 \beta dL_g = \frac{ \sum_{k \in I_i \cap \{1,\cdots,m\}}  \partial_k f( \bar{\sigma}_1,\cdots, \bar{\sigma}_m )}{ \sum_{k=1}^m \sigma_k \partial_k f( \bar{\sigma}_1,\cdots, \bar{\sigma}_m )  } \int_{\partial \Sigma}\beta dL_g.$$
\end{itemize}
where $\bar{\sigma}_i = \sigma_i \int_{\partial \Sigma}\beta dL_g$. If $\beta$ is critical for the functional $E_f^S(\Sigma,g,\cdot)$ restricted to the conformal class of $g$, then there is a map $\Phi : \Sigma \to \R^n$ that satisfies (1), (2) and (4).
\end{theo}
Notice that (1) and (2) mean that $\Phi$ is a free boundary harmonic map into an ellipsoid and that (1), (2) and (3) mean that $\Phi$ is a (branched) minimal immersion into an ellipsoid. The latter condition (4) comes from a chain rule in the computation of critical points for $f \circ (\bar{\sigma}_1,\cdots,\bar{\sigma}_m)$. We proved in \cite[Theorem 4.5]{PetridesTewodrose} that any branched free boundary minimal immersion into an ellipsoid can be seen as a critical point of $E_f^S(\Sigma,\cdot,\cdot)$ for several families of functions $f$ that satisfy (4). Generally speaking, a classification of free boundary minimal surfaces into ellipsoids would provide a great progress to look for optimal constants associated to isoperimetric inequalities involving Steklov eigenvalues. With the notations of Theorem \ref{theo:critcombS}, we say that $\Phi$ is associated to the critical point $(g,\beta)$ or to the critical metric $\hat{\beta}^2 g$ (for some extension $\hat{\beta}$ of $\beta$ in $\Sigma$), or that $(g,\beta)$ and $\hat{\beta}^2 g$ are associated to $\Phi$.

In the current paper we give examples of critical points of functionals involving two Steklov eigenvalues on the disk. In particular, we classify critical points of functionals involving the two first Steklov eigenvalues on the disk associated to a free boundary minimal immersion into an ellipse (see Proposition \ref{prop:plongéplan}: the map $\Phi$ given by Theorem \ref{theo:critcombS} has two coordinates). From that classification, we deduce

\begin{prop} \label{prop:nonplanarbifurcation} We assume that $t > 3$. Then $I^S(\mathbb{D},h_t^+) < \frac{(3+t)\sqrt{3}}{6\pi}$
and $I^S(\mathbb{D},h_t^+)$ has to be realized by a metric associated to non-planar immersed free boundary minimal disks into an ellipsoid of $\R^3$.
\end{prop}
The novelty compared with the main result in \cite{Petridesnonplanardisk} is the rank $t_0 = 3$. It appears as a bifurcation point from planar critical ellipses to non-planar critical free boundary minimal surfaces into ellipsoids (see the proof of Proposition \ref{prop:condnesscrit}). From this observation, we state the following conjecture:
\begin{conj} \label{conj} Let $1 \leq t \leq 3$. Then
$$ I^S(\mathbb{D},h_t^+) = 
\frac{\sqrt{t}}{2\pi} $$
Moreover, it is uniquely attained by the critical ellipse $co(E_t) = \{ x^2 + t y^2 \leq 1\}$ endowed with the Euclidean metric $\xi$ and the weight $\beta(x,y) = (x^2 + t y^2)^{-1}$ on the boundary $E_t = \{ x^2 + t y^2 = 1\}$.
\end{conj}
Together with \eqref{eq:herschpayne} and Proposition \ref{prop:nonplanarbifurcation}, this conjecture would improve the description of $I^S(\mathbb{D},h_t^+)$ and its minimizers.
This conjecture would be related to the following conjecture that is natural thanks to our new description of the critical ellipses.
\begin{conj} \label{conj} 
Let $q \geq 1$. There is a non planar free boundary minimal immersed disk into the ellipsoid $\{ x^2 + q y^2 + qz^2  \leq 1\}$ if and only if $q > 3$.
\end{conj}
This conjecture is already true for $q=1$ by  \cite{fs} (see Theorem \ref{thm_fs} below). The "if" part of this conjecture would be true if we prove that the non planar disks of Proposition \ref{prop:nonplanarbifurcation} behave continuously with respect to $t$ from the bifurcation point at $t=3$.

\medskip

In a second part of the paper, we provide the following expected result that is not written elsewhere to our knowledge.
\begin{theo} \label{theo:largeineq}
Let $(\Sigma,g)$ be a Riemannian surface with boundary. Then
\begin{equation} \label{eq:ineqmain} I_c^S(\Sigma,[g],f) \leq I_c^S(\tilde{\Sigma}  ,[\tilde{g}],f) , \end{equation}
for any 
$$ (\tilde{\Sigma}, \tilde{g}) = (\Sigma,g) \sqcup (\mathbb{D},\xi) \sqcup \cdots \sqcup (\mathbb{D},\xi) $$
or
$$ (\tilde{\Sigma}, \tilde{g}) = (\mathbb{D},\xi) \sqcup \cdots \sqcup (\mathbb{D},\xi), $$
where $\xi$ is the Euclidean metric on the copies of disks.
\end{theo}
This result is a natural counterpart of \cite{CES} for Laplace eigenvalues on closed surfaces. We also refer to \cite{fs3} for related results. Theorem \ref{theo:largeineq} was stated in \cite{petrides-3} without proof and is used in \cite{PetridesVarMethod}. Moreover, in \cite{PetridesCS,PetridesVarMethod}, we proved that if the inequality \eqref{eq:ineqmain} is strict for any $(\tilde{\Sigma},\tilde{g})$ given by Theorem \eqref{theo:largeineq}, then $I_c^S(\Sigma,[g],f)$ is attained. Proving such strict inequalities is much harder than Theorem \ref{theo:largeineq}. We provide examples where these strict inequalities hold in \cite[Theorem 0.2]{Petridesnonplanardisk}:
$$ I^S(\mathbb{D} , h_t^\pm) < I^S(\mathbb{D} \sqcup \mathbb{D} , h_t^\pm ).$$
As already said, this implies the existence of minimizers for $I^S(\mathbb{D} , h_t^\pm)$. However, the simplest strict inequality is still a conjecture:
\begin{conj} \label{eq:conjstrict}
\begin{equation*}  \sigma_1(\Sigma,[g]) > 2\pi = \sigma_1(\mathbb{D}) \end{equation*}
for any compact surface with boundary $(\Sigma,g)$ that is not diffeomorphic to the disk.
\end{conj}
The elaborate techniques of the author from \cite{PetridesVarMethod,PetridesNew} and their extension in the present paper can be used to prove this conjecture in all the conformal classes of the annulus and the M{\"o}bius band.
\begin{theo} \label{thm_annuli}
Let $\Sigma$ be an annulus or a M{\"o}bius band endowed with a flat metric $g$ then
there is a smooth positive function $\beta : \partial\Sigma \to \R_+^{\star}$ such that
\begin{equation} \label{eq_gap}
\bar{\sigma}_1 \left(\Sigma,g,\beta \right) > 2 \pi.
\end{equation}
\end{theo}

As a consequence of the work in \cite{petrides-3} and Theorem \ref{thm_annuli} we also obtain
\begin{theo} \label{thm_harmonic}
Let $\Sigma$ be an annulus or a M{\"o}bius band endowed with a flat metric $g$ then there 
there is a smooth positive function $\beta : \partial\Sigma \to \R_+^{\star}$ such that
\begin{equation} \label{conjecture:strict}
\sigma_1 \left(\Sigma,g,\beta \right) = \sigma_1(\Sigma,[g]).
\end{equation}
In particular, there is a free boundary harmonic map $\Phi \colon (\Sigma,g) \to \mathbb{B}^N$ by first eigenfunctions.
Here, $N \leq 3$ if $\Sigma$ is an annulus and $N \leq 4$ if $\Sigma$ is a M{\"o}bius band.
\end{theo} 
Notice that in \cite{petrides}, we prove the analagous strict inequality in the context of Laplace eigenvalues on connected closed surfaces, and as a consequence the existence of maximizers for the first Laplace eigenvaule (renormalized by the area) in any conformal class. This fundamental result is also used in \cite{KS}.
The technical extension of \cite{PetridesVarMethod,PetridesNew} required to prove Theorem \ref{thm_annuli} is the following.
\begin{prop} \label{prop:construction>2pi} Let $\Sigma_{{l},\eps}^{\pm}$ be the compact surfaces with boundary defined by \eqref{def:sigmalepspm}. Then, for any $l>0$, we have that
\begin{equation} \label{eq:construction>2pi} \bar{\sigma}_k(\Sigma_{{l},\eps}^{\pm},[g_\eps] ) >2\pi \end{equation}
holds for $\eps$ small enough.
\end{prop}
Roughly speaking, Proposition \ref{prop:construction>2pi} completes the proof of the conjecture \ref{eq:conjstrict} close to the boundary of the moduli space of Annuli and M\"obius bands. More generally, we strongly believe that the techniques of \cite{PetridesVarMethod,PetridesNew} could be used to prove that conjecture \ref{eq:conjstrict} holds outside a compact subset of the Teichmuller space of any compact surface with boundary. That result could be a first step to complete the proof of the conjecture. By the way, Proposition \eqref{prop:construction>2pi} leads to the proof of Theorem \ref{thm_annuli}. The method we use is very specific to annuli and M\"obius bands since their moduli spaces are homeomorphic to a line and critical points of the first Steklov eigenvalues on these surfaces are unique by \cite{fs} (see Theorem \ref{thm_fs} below).

The paper is organized as follows:
in Section \ref{sec1}, we compute the set of eigenvalues of the new critical ellipses and give consequences on isoperimetric inequalities involving Steklov eigenvalues. In Section \ref{sec2}, we prove Theorem \ref{theo:largeineq}. In Section \ref{sec3}, we prove Proposition \ref{prop:construction>2pi} that is crucial to prove Theorem \ref{thm_annuli}. Appendix \ref{sec4} is devoted to the proof of Theorem \ref{thm_annuli}: it is extracted from the preprint \cite{MP} and is written in collaboration with Henrik Matthiesen.

\section{Critical ellipses} \label{sec1}

For $q\geq 1$, we denote by $co(E_q) = \{ x^2 +q y^2 \leq 1 \}$ the convex hull of the ellipse
$$E_q = \{x^2 + qy^2 = 1\}$$
endowed with the Riemannian metric $g_q = \beta_q^2 (dx^2 +dy^2)$ such that $$\beta_q = (x^2+q^2y^2)^{-\frac{1}{2}}$$ 
on $E_q$. Notice that the parametrization $(\cos\theta,\frac{1}{\sqrt{q}}\sin\theta)$ of the ellipse gives the length measure associated to the Euclidean metric on the boundary
\begin{equation}\label{eq:lengthmeasure}dL_{E_q} = \sqrt{\sin^2\theta+ \frac{1}{q}\cos^2\theta}d\theta = \frac{1}{\sqrt{q}} \beta_q^{-1}d\theta.\end{equation}
Then, the total length for $g_q$ is 
\begin{equation} \label{eq:totallength} L_{E_q}(\beta_qdL_{g_q}) = \int_{E_q} \beta_q dL_{E_q} = \int_{0}^{2\pi} \frac{d\theta}{\sqrt{q}} = \frac{2\pi}{\sqrt{q}}. \end{equation}

In the following result, we prove that $(co(E_q),g_q)$ is critical for numerous combinations of Steklov eignvalues. Then, it is a good candidate to extremize sharp isoperimetric inequalities involving Steklov eigenvalues. If $\sigma$ is a Steklov eigenvalue of $(\Sigma,g)$, the index of $\sigma$ is the smallest integer $k$ such that $\sigma_k(\Sigma,g) = \sigma$.

\begin{theo} \label{theo:classellipse} Let $q\geq 1$. The ellipse $(co(E_q),g_q)$ has the following set of eigenfunctions $(Re(P_n^q),Im(P_n^q))$ and associated eigenvalues $(\sigma_n^q,\tau_n^q)$:
$$ \sigma_n^q  = n \sqrt{q} \frac{(\sqrt{q}+1)^n-(\sqrt{q}-1)^n}{(\sqrt{q}+1)^n+(\sqrt{q}-1)^n} \text{ and } \tau_n^q = n \sqrt{q} \frac{(\sqrt{q}+1)^n+(\sqrt{q}-1)^n}{(\sqrt{q}+1)^n-(\sqrt{q}-1)^n}$$
and
$$ P_n^q(z) = \frac{1}{2^{n-1}} \sum_{k=0}^{\left[\frac{n}{2}\right]} {n \choose 2k} z^{n-2k} \left(z^2 - \left(1-\frac{1}{q}\right) \right)^k.$$
Moreover, for $q>1$, $g_q$ is a critical metric for all functionals $g\mapsto f(\bar{\sigma}_{k_1}(g),\bar{\sigma}_{k_2}(g))$ such that
$$ \frac{ \partial_1 f(\bar{\sigma}_{k_1}(g_q),\bar{\sigma}_{k_2}(g_q))}{\partial_2 f(\bar{\sigma}_{k_1}(g_q),\bar{\sigma}_{k_2}(g_q))} = \frac{\int_{E_q} Re(P_{n}^q)^2 dL_{g_q}}{\int_{E_q} Im(P_{n}^q)^2 dL_{g_q}}, $$
where $(k_1,k_2)$ is the couple of indices of the couple of distinct eigenvalues $(\sigma_n^q,\tau_n^q)$.
\end{theo}

\begin{proof}
We assume that $P_n$ is a unit polynomial of degree $n$ such that $Re(P_n)$ and $Im(P_n)$ are Steklov eigenfunctions with eigenvalues $\sigma_n$ and $\tau_n$ in $E_q$. As real and imaginary parts of complex polynomials, they are harmonic. We compute the equation on the boundary $E_q$: we have for $z = x+iy \in E_q$,
$$ \partial_\nu P_n(z) = \left\langle \nabla P_n(z) , \nu(z) \right\rangle = \left\langle \left( P_n'(z) , iP_n(z) \right) , \frac{(x, qy) }{(x^2+q^2 y^2)^{\frac{1}{2}}} \right\rangle = \frac{ (x+iqy)P_n'(z) }{(x^2+q^2 y^2)^{\frac{1}{2}}}$$
so that
\begin{equation} \label{eq:asatisfairepnq} (x+iqy)P_n'(z) = \sigma_n Re(P_n) + i \tau_n Im(P_n).\end{equation}

We prove that $P_n^q$ satisfies \eqref{eq:asatisfairepnq}. Let $(x,y) = (\cos t , i \frac{\sin t}{\sqrt q}) \in E_q$. We set
$$ A_n^q = \frac{ (1+\frac{1}{\sqrt{q}})^n + (1-\frac{1}{\sqrt{q}})^n }{2^n} \text{ and } B_n^q = \frac{ (1+\frac{1}{\sqrt{q}})^n - (1-\frac{1}{\sqrt{q}})^n }{2^n} .$$
By \cite[Theorem 1 (with $a=1$, $b=\frac{1}{\sqrt{q}}$)]{Nikiforova} we deduce
\begin{equation} \label{eq:relpnq} P_n^q\left( \cos t + i \frac{\sin t}{\sqrt{q}} \right) = A_n^q \cos nt + i B_n^q \sin nt. \end{equation}
We obtain the derivative with respect to $t$
$$ \left( -\sin t + i \frac{\cos t}{\sqrt q} \right) (P_n^q)'\left( \cos t + i \frac{\sin t}{\sqrt{q}} \right) = n \left( - A_n^q \sin nt + i B_n^q \cos nt\right) $$
that gives if we set $z=x+iy$:
$$ \frac{i}{\sqrt q}\left( \cos t +iq \frac{\sin t}{\sqrt q} \right)(P_n^q)'(z) = n \left( - A_n^q \sin nt + i B_n^q \cos nt  \right) $$
so that
$$ \left( x+iqy \right)(P_n^q)'(z) = \sqrt{q} n \left( B_n^q \cos nt  + i A_n^q \sin nt \right) $$
and we deduce from this equality and \eqref{eq:relpnq} that:
$$ \left( x+iqy \right)(P_n^q)'(z) = \sqrt{q} n \frac{ B_n^q}{A_n^q} Re\left(P_n^q(z)\right)  + i \sqrt{q} n \frac{ A_n^q}{B_n^q} Im\left(P_n^q(z)\right). $$

Moreover, the constant functions and $(Re(P_n^q),Im(P_n^q))_{}$ is the whole sequence of Steklov eigenfunctions of $(co(E_q), g_q)$. Indeed, up to renormalization, it is a Hilbert basis of $L^2(E_q,g_q)$ since with \eqref{eq:relpnq} on $E_q$, they correspond to harmonic polynomials.

The last part of the proposition comes from
$$ \sigma_n^q Re(P_n^q)^2 + \tau_n^q Im(P_n^q)^2 = c_{n,q} := \sqrt{q}n A_{n,q}B_{n,q} \text{ sur } E_q. $$ 
It implies that $(c_n^q)^{-\frac{1}{2}} P_{n,q}$ is a free boundary conformal harmonic map into
$$\{ \sigma_n^q x^2 + \tau_n^q y^2 \leq 1 \}. $$
Therefore, it provides critical metrics of spectral functionals: we apply \cite[Theorem 4.5]{PetridesTewodrose} to complete the proof.
\end{proof}

For $q = 1$, $P_n^q = z^n$ correspond to Steklov eigenfunctions on the Euclidean disk : the homogeneous harmonic polynomials. To our knowledge, these Riemannian surfaces are new examples of compact simply connected surfaces with boundary 
that are critical for spectral functionals. We emphasize that the order of eigenvalues on the disks
$$ 1 = \sigma_1^1 = \tau_1^1 < 2 = \sigma_2^1 = \tau_2^1 < \cdots < n = \sigma_2^n = \tau_2^n < \cdots $$
is not the same anymore as $q$ grows. Indeed, we have
$$\forall q \in [1,+\infty[, (\sigma_n^q)_{n\geq 1} \text{ and } (\tau_n^q)_{n\geq 1} \text{ are increasing}$$
$$ \forall n \in \mathbb{N}^*, (\sigma_n^q)_{q\geq 1} \text{ is bounded and } \tau_n^q \sim q \text{ as } q \to +\infty.   $$

\medskip

As examples, let's focus on spectral functionals on the disk that combine first and second Steklov eigenvalues. Let $f : \R^2_+ \to \R \cup \{\infty\}$ be such that $\partial_i f \leq 0$ for $i=1,2$ and
$$ E_f^S(\mathbb{D},g) = F(\bar{\sigma}_1(\mathbb{D},g), \bar{\sigma}_2(\mathbb{D},g)).$$
In the following proposition, we prove that the only critical points of $E_f^S(\mathbb{D},\cdot)$ that are metrics associated to branched minimal immersions into ellipses have to be Steklov isometric to $(E_q,g_q)$ for $q\geq 1$. We denote $\mathcal{E}_\sigma = \{ \sigma_1 x_1^2 + \sigma_2 x_2^2 = 1 \}$.

\begin{prop}[{\cite[Proposition 1.2]{Petridesnonplanardisk}}] \label{prop:plongéplan}
Let $\Phi : \mathbb{D} \to co\left(\mathcal{E}_{\sigma}\right) \subset \R^2$ be a conformal free boundary harmonic map into an ellipse. We assume that the coordinates of $\Phi$ are first or second Steklov eigenfunctions with respect to the metric 
$$g = e^{2v} \Phi^\star \xi \text{ such that } e^v = (\sigma_1^2(\phi_1)^2 + \sigma_2^2 (\phi_2)^2  )^{-\frac{1}{2}}  \text{ on }  \mathbb{S}^1 $$
where $\xi$ is the Euclidean metric on $co\left(\mathcal{E}_{\sigma}\right)$. Then $\Phi$ is a biholomorphism. Then, there is $q>1$ such that $(\mathbb{D},g)$ is Steklov-isometric to $(E_q,g_q)$ (up to dilatation).
\end{prop}

\begin{proof} Notice that as a conformal harmonic map, $\Phi$ is holomorphic. It remains to prove that $\Phi$ is a difféomorphism from the closed disk to the closed ellipse.

\medskip

\noindent \textbf{Step 1}: $e^{v}$ does not vanish by application of the maximum principle and Hopf lemma (see \cite[Claim 1.1]{Petridesnonplanardisk}): by conformality, $\Phi_{\vert \mathbb{S}^1} : \mathbb{S}^1 \to \mathcal{E}_{\sigma}$ is an immersion.

\medskip

\noindent \textbf{Step 2}: One coordinate, for instance $\phi_1$ is a first Steklov eigenfunction. $\phi_1$ has only two nodal domains. Then, the nodal line of $\phi_1$ is a connected line that ends at the boundary $\mathbb{S}^1$. Then, the degree of $\Phi_{\vert \mathbb{S}^1} : \mathbb{S}^1 \to  \mathcal{E}_{\sigma}$ is $1$.

\medskip

\noindent \textbf{Step 3}: A holomorphic map $\Phi$ between two simply connected domains such that the restriction to the boundary has degree $1$ has to be a biholomorphism.
\end{proof}

\begin{remark} To complete the classifictaion, it would be interesting to know if any conformal free boundary harmonic map from a disk into an ellipse has to be associated to a metric $g$ such that $(\mathbb{D},g)$ is Steklov isometric to $(E_q,g_q)$ (up to dilatation).
\end{remark}

\medskip

We apply these results to functions $f = h_{s,t}$ of \cite{Petridesnonplanardisk} for $s > 0$ and $t \geq 0$.
$$ h_{s,t}(x_1,x_2) = \left(x_1^{-s} + t x_2^{-s}\right)^{\frac{1}{s}}  $$
We obtain the following property:

\begin{prop} \label{prop:condnesscrit}
Assume that for $q>1$, $(E_q,g_q)$ is critical for $E_f^S(\mathbb{D},\cdot)$. Then:
\begin{equation}\label{eq:critEFS} q \leq 3 \text{ and }  \frac{ \partial_1 f(\frac{2\pi}{\sqrt{q}},2\pi\sqrt{q})}{\partial_2 f(\frac{2\pi}{\sqrt{q}},2\pi \sqrt{q})} =q. \end{equation}
In particular, let $g_{s,t}$ be a minimizer of $E_{h_{s,t}}^S(\mathbb{D},\cdot)$. We assume that
$$ s>0 \text{ and } t > 3^s \text{ or } s<0 \text{ and } t \geq ( 2^{-s}-1)^{-1}. $$
Then, every free boundary minimal immersion by  first and second eigenfunctions associated to the metric $g_{s,t}$ is non planar.
\end{prop}

\begin{proof}
We prove \eqref{eq:critEFS}. Notice that on $(E_q,g_q)$,
\begin{equation} \label{eq:eigenvaluesigma1tau1} \bar{\sigma}_1(g_q) = \sigma_1^q = 1 \text{ and } \tau_1^q = q.\end{equation}
Then, a necessary condition to make the ellipse $(E_q,g_q)$ of Theorem \ref{theo:classellipse} critical for $E_F^S$ is
$\sigma_2(E_q,g_q) = q$. It is true if and only if $1\leq q \leq 3$. Indeed, $q=3$ corresponds to the first bifurcation point $\sigma_2^q = \tau_1^q$.

\medskip

We now compute the necessary condition (4) of Theorem \ref{theo:critcombS}. We recall that the length measure associated to $g_q$ is given by $\beta_q dL_{E_q}$ where $dL_{E_q}$ is computed in \eqref{eq:lengthmeasure}. Then, the masses of the coordinates are
$$ \int_{E_q} x^2 dL_{g_q} = \int_0^{2\pi} \cos^2(\theta) \frac{d\theta}{\sqrt{q}} = \frac{\pi}{\sqrt{q}} \text{ and } \int_{E_q} y^2 dL_{g_q} = \int_0^{2\pi} \frac{1}{q}\cos^2(\theta) \frac{d\theta}{\sqrt{q}} = \frac{\pi}{q\sqrt{q}}.$$
Then, using \eqref{eq:totallength} and \eqref{eq:eigenvaluesigma1tau1} for $1\leq q \leq 3$, we have
$$ \bar{\sigma}_1(E_q,g_q) = \frac{2\pi}{\sqrt{q}} \text{ and }  \bar{\sigma}_2(E_q,g_q) = 2\pi\sqrt{q}. $$
We then deduce the condition \eqref{eq:critEFS} from Theorem \ref{theo:critcombS}.

We deduce from \eqref{eq:critEFS} computed for $h_{s,t}$ :
$$  t  = q^s. $$
For $t > 3^s$, we deduce that there is not any ellipse $(E_q,g_q)$ for $q>1$ that is critical for $E_{h_s,t}^s$.
Moreover, we have
\begin{equation*} \begin{split} & E_{h_{s,t}}^s( E_{t^{\frac{1}{s}}}, g_{t^{\frac{1}{s}}} ) = h_{s,t}\left(\frac{2\pi}{t^{\frac{1}{2s}}},2\pi t^{\frac{1}{2s}} \right)  = (2\pi)^{-1} \left( 2 \sqrt{t} \right)^{\frac{1}{s}} \\ & E_{h_{s,t}}^s(\mathbb{D} ) = h_{s,t}\left(2\pi,2\pi\right)  = (2\pi)^{-1} \left( 1+t \right)^{\frac{1}{s}}  \end{split} \end{equation*}
Then, the Euclidean disk is not a minimiser for $s>0$, and $(E_q,g_q)$ is not a minimizer. Finally, for $s<0$, we have by \cite[Theorem 0.2]{Petridesnonplanardisk} that
$$ I_{h_{s,t}}^s(\mathbb{D} ) < I_{h_{s,t}}^s(\mathbb{D} \sqcup \mathbb{D} ) = h_{s,t}\left(0,4\pi\right) = (4\pi)^{-1} t^{\frac{1}{s}}.  $$
However, 
$$ (2\pi)^{-1}(1+t)^{\frac{1}{s}} < (4\pi)^{-1} t^{\frac{1}{s}} \Leftrightarrow t < ( 2^{-s}-1)^{-1}.$$
We completed the proof of the proposition.
\end{proof}

Finally, we notice that for any $n$ and $q$, the quantity
$$ \sigma_n^q \tau_n^q L_{g_q}(E_q)^2 =  4\pi^2 n^2 $$
is independent of $q$. Then, we immediately obtain new minimizers of the Hersch-Payne-Schiffer functional $f_{1,2}(x) = (x_1 x_2)^{-1}$:
\begin{prop} \label{prop:crithps}
For $1\leq t \leq 3$ critical ellipses $co(E_t) = \{ x^2 + t y^2 \leq 1\}$ endowed with the Euclidean metric $\xi$ and the weight $\beta_t(x,y) = (x^2 + t y^2)^{-1}$ on the boundary $E_t = \{ x^2 + t y^2 = 1\}$ are minimizers of $E_{f_{1,2}}^S(\mathbb{D},\cdot)$. Moreover, they are the only minimizers of $E_{f_{1,2}}^S(\mathbb{D},\cdot)$ associated to a free boundary minimal immersion into an ellipse.
\end{prop}
Notice that by Theorem \ref{theo:classellipse}, computations given in the proof of Proposition \ref{prop:condnesscrit} show that $(co(E_q),\xi,\beta_q)$ are still critical points or $E_{f_{1,k}}^S(\mathbb{D},\cdot)$, where $k$ is the index of $\tau_1^q$.

\section{Inequalities on conformal invariants} \label{sec2}
We aim at proving Theorem \ref{theo:largeineq}. We apply the following proposition:

\begin{prop} \label{prop:convmetrictest}  Let $\beta_0: \Sigma \to \R_+$ be a function. We assume that $\beta_0=0$ or that $\beta_0$ is a positive function. Let $\beta_i : \mathbb{D} \to \R_+$ for $i \in \{1,\cdots,k\}$ be positive functions. There is a sequence $\beta_\eps : \Sigma \to \R_+$ of weights on $(\Sigma,g)$ that satisfy for any $m \in \mathbb{N}^*$,
$$ \bar{\sigma}_m(\Sigma,g, \beta_\eps) \to \bar{\sigma}_m( \tilde{\Sigma},\tilde{g}, \tilde{\beta})   $$
as $\eps \to 0$, where $(\tilde{\Sigma},\tilde{g}, \tilde{\beta}) = (\Sigma,g,\beta_0) \sqcup (\mathbb{D},\xi,\beta_1) \sqcup \cdots \sqcup (\mathbb{D},\xi,\beta_k)$
\end{prop}

\begin{proof}[Proof of Theorem \ref{theo:largeineq} with Proposition \ref{prop:convmetrictest}]
Let $\delta>0$. Let $\beta_0: \Sigma \to \R_+$ a non-negative smooth function and $\beta_i : \mathbb{D} \to \R_+$ for $i= \{1,\cdots,k\}$ be non-zero non-negative smooth functions such that
$$ E_f^S( \tilde{\Sigma},\tilde{g}, \tilde{\beta}  ) \leq I_c^S(\tilde{\Sigma}  ,[\tilde{g}],f) + \delta$$
where if $\beta_0 \neq 0$,
$$ (\tilde{\Sigma},\tilde{g}, \tilde{\beta}) = (\Sigma,g,\beta_0) \sqcup (\mathbb{D},\xi,\beta_1) \sqcup \cdots \sqcup (\mathbb{D},\xi,\beta_k) $$
and if $\beta_0 = 0$,
$$ (\tilde{\Sigma},\tilde{g}, \tilde{\beta}) = (\mathbb{D},\xi,\beta_1) \sqcup \cdots \sqcup (\mathbb{D},\xi,\beta_k). $$
From Proposition \ref{prop:convmetrictest}, we obtain a weight $\beta_\eps$ such that 
$$f(\bar{\sigma}_1(\Sigma,g,\beta_\eps),\cdots,\bar{\sigma}_m(\Sigma,g,\beta_\eps)) \to f(\bar{\sigma}_1(\tilde{\Sigma},\tilde{g},\tilde{\beta}),\cdots,\bar{\sigma}_m(\tilde{\Sigma},\tilde{g},\tilde{\beta})) $$
as $\eps \to 0$. Testing $\beta_\eps$ in the variational characterization of $I_c^S(\Sigma,[g],f)$, we obtain 
$$ I_c^S(\Sigma,[g],f) \leq \liminf_{\eps \to 0} E_f^S(\Sigma,g,\beta_\eps) \leq E_f^S(\tilde{\Sigma},\tilde{g},\tilde{\beta}) \leq I_c^S(\tilde{\Sigma},[\tilde{g}],f) + \delta.$$
Letting $\delta \to 0$ ends the proof of the theorem.
\end{proof}

Therefore, the section is now devoted to the proof of Proposition \ref{prop:convmetrictest}. 
We let $x_1,\cdots,x_k \in \Sigma$ be distinct points and 
$\Omega_1,\cdots,\Omega_k$ be pairwise disjoint open neighborhoods of $x_1,\cdots,x_k$ such that there are diffeomorphisms 
$\theta_i : \Omega_i \to \mathbb{D}_{\delta}^+$ that satisfy
$$ \theta_i(\Omega_i \cap \partial \Sigma) = [-\delta,\delta] \times \{0\} $$
$$ \theta_i^\star( e^{2u_i} \xi ) = g $$ 
$$ \theta_i(x_i) = 0.$$
We now let $\zeta : \mathbb{D} \to \mathbb{R}^2_+$ be the biholomorphism defined by $\zeta(1) = 0$, $\zeta(-1) = \infty$, $\zeta(0) = i$. If $x \in \Omega_i$, we denote $z^i = \theta_i(x)$.
$$ \beta_\eps(x) = \beta_0(x) + \sum_{i=1}^k  \eta(z^i)  e^{-u_i(z^i)} \beta_i\left( \zeta^{-1}\left(  \frac{z^i}{\eps}\right) \right) \frac{2 \eps}{(\eps + z^i_2)^2 + \left(z^i_1\right)^2}, $$
where $\eta \in \mathcal{C}^\infty_c\left( \mathbb{D}_\delta\right)$ satisfies $0 \leq \eta \leq 1$ and $\eta = 1$ in $\mathbb{D}_{\frac{\delta}{2}}$.

We first prove the following 
\begin{cl} \label{cl:boundedsigma}
$$ \bar{\sigma}_m(\Sigma,g, \beta_\eps) \leq \bar{\sigma}_m( \tilde{\Sigma},\tilde{g}, \tilde{\beta}  ) + O\left(\frac{1}{\ln\eps}\right) $$
as $\eps \to 0$.
\end{cl}

\begin{proof}
Let $(\varphi_0,\cdots, \varphi_m)$ be a $L^2(\partial \tilde{\Sigma},\tilde{\beta}dL_{\tilde{g}})$-orthonormal family of $m+1$ first eigenfunctions associated to $(\tilde{\Sigma},\tilde{g}, \tilde{\beta})$. From these functions, we construct test functions for $\bar{\sigma}_m(\Sigma,g, \beta_\eps)$. We set the cut-off functions with disjoint support
$$ \chi_0^\eps(x) = 1 - \sum_{i=1}^k \eta_{\eps^{\frac{1}{8}}}(z^i) \text{ and } \chi_i^\eps(x) = \eta_{\eps^\frac{1}{4}}(z^i),$$
where $\eta_{r} \in \mathcal{C}^\infty_c(\mathbb{D}_r)$ is such that $0 \leq \eta_r \leq 1$, $\eta_r = 1$ in $\mathbb{D}_{r^2}$ and 
$$ \int_{\mathbb{D}_r} \vert \nabla \eta_r \vert^2_\xi dA_\xi \leq 
\frac{C}{\ln \frac{1}{r}} $$
and we obtain the following test functions for $1 \leq j \leq m$:
$$ \varphi_\eps^j(x) = \chi_0^\eps(x) \left(\varphi_j\right)_{\vert \Sigma}(x) + \sum_{i=1}^k \chi_i^\eps(x) \left(\varphi_j\right)_{\vert D_i}\left( \zeta^{-1}\left( \frac{z^i}{\eps} \right) \right).$$
We then have the existence of $a = (a_0,\cdots,a_m) \in \mathbb{S}^m$ such that
\begin{equation} \label{eq:eqtestpsieps} \sigma_m(\Sigma,g, \beta_\eps) \leq \frac{\int_\Sigma \left\vert \nabla \psi_\eps \right\vert_g^2 dA_g}{\int_{\partial \Sigma} \psi_\eps^2 \beta_\eps dL_g} \text{ where } \psi_\eps = \sum_{j=0}^m a_\eps^j \varphi_\eps^j .\end{equation}
Now we set $T_\eps : \tilde{\Sigma} \to \R$ the function defined by:
$$T_\eps = \sum_{j=0}^m a_j^\eps \left(\varphi_j\right). $$
that is cutted in evry connected part as $\psi_0^\eps : \Sigma \to \R$ :
$$ \psi_0^\eps = \chi_0^\eps (T_\eps)_{\vert \Sigma} $$
and as $\psi_i^\eps : \mathbb{D}\to \R$ for $1 \leq i \leq k$ and $z \in \mathbb{D}$,
$$\psi_i^\eps(z) = \tilde{\chi}_i^\eps (T_\eps)_{\vert D_i} $$
where we set $\tilde{\chi}_i^\eps(z) = \eta_{\eps^{\frac{1}{4}}}\left( \eps \zeta(z)  \right) $
All functions $\chi_i^\eps$ have disjoint support. We then have that
$$ \int_\Sigma \left\vert \nabla \psi_\eps \right\vert_g^2 dA_g = \int_\Sigma \left\vert \nabla \psi_0^\eps \right\vert_g^2 dA_g + \sum_{i=1}^k \int_{\mathbb{D}} \left\vert \nabla \psi_i^\eps \right\vert_\xi^2 dA_\xi $$
and that
$$ \int_{\partial \Sigma} (\psi_\eps)^2 \beta_\eps dL_g = \int_{\Sigma} (\psi_0^\eps)^2 \beta_\eps dL_g + \sum_{i=1}^k \int_{\mathbb{R}\times \{0\}} \psi_i^\eps( \zeta^{-1}(s))^2 \beta_\eps( \theta_i^{-1}( \eps  s  ) ) \eps ds. $$
Knowing that $\vert \nabla(T_\eps)_{\vert \Sigma} \vert^2_g$ is uniformly bounded,  we compute
\begin{align*} \int_\Sigma \left\vert \nabla \psi_0^\eps \right\vert_g^2 dA_g = & \int_\Sigma \vert \nabla(T_\eps)_{\vert \Sigma} \vert^2_g dA_g +  \int_\Sigma ((\chi_0^\eps)^2-1) \vert \nabla(T_\eps)_{\vert \Sigma} \vert^2_g dA_g  \\
&+ 2 \int_\Sigma \chi_0^\eps (T_\eps)_{\vert \Sigma} \langle \nabla \chi_0^\eps , \nabla (T_\eps)_{\vert \Sigma} \rangle_g dA_g  +  \int_\Sigma \vert \nabla \chi_0^\eps  \vert^2_g \left((T_\eps)_{\vert \Sigma}\right)^2 dA_g \\
=&  \int_{\Sigma} \vert \nabla(T_\eps)_{\vert \Sigma} \vert^2_g dA_g + O(\eps^{\frac{1}{4}}) + O\left( \eps^{\frac{1}{8}} \left(\ln\frac{1}{\eps}\right)^{-\frac{1}{2}} \right) + O\left( \left(\ln\frac{1}{\eps}\right)^{-1}\right) \end{align*}
as $\eps \to 0$ and knowing that  $\vert \nabla(T_\eps)_{\vert D_i} \vert^2_\xi$ is uniformly bounded, 
\begin{align*} \int_{\mathbb{D}} \left\vert \nabla \psi_i^\eps \right\vert_\xi^2 dA_\xi = & \int_{\mathbb{D}} \vert \nabla(T_\eps)_{\vert D_i} \vert^2_\xi dA_\xi + \int_{\mathbb{D}} ((\tilde{\chi}_i^\eps)^2-1) \vert \nabla(T_\eps)_{\vert D_i} \vert^2_\xi dA_\xi \\ & + 2 \int_{\mathbb{D}} \tilde{\chi}_i^\eps (T_\eps)_{\vert D_i} \langle \nabla \chi_i^\eps , \nabla (T_\eps)_{\vert D_i} \rangle_\xi dA_\xi  +  \int_{\mathbb{D}} \vert \nabla \tilde{\chi}_i^\eps  \vert^2_\xi \left((T_\eps)_{\vert D_i}\right)^2 dA_\xi \\
=&  \int_{\mathbb{D}} \vert \nabla(T_\eps)_{\vert D_i} \vert^2_\xi dA_\xi + O(\eps^{\frac{1}{2}}) + O\left( \eps^{\frac{1}{4}} \left(\ln\frac{1}{\eps}\right)^{-\frac{1}{2}} \right) + O\left( \left(\ln\frac{1}{\eps}\right)^{-1}\right) \end{align*}
as $\eps \to 0$. The previous equalities lead to
\begin{equation} \label{eq:eqgradpsieps} \int_\Sigma \left\vert \nabla \psi_\eps \right\vert_g^2 dA_g =\int_{\Sigma} \vert \nabla(T_\eps)_{\vert \Sigma} \vert^2_g dA_g + \sum_{i=1}^k \int_{\mathbb{D}} \vert \nabla(T_\eps)_{\vert D_i} \vert^2_\xi dA_\xi +  O\left( \left(\ln\frac{1}{\eps}\right)^{-1}\right)\end{equation}
as $\eps\to 0$. We also compute
\begin{align*} \int_{\partial\Sigma} (\psi_0^\eps)^2 \beta_\eps dL_g = \int_{\partial \Sigma} \left(T_\eps\right)_{\vert \Sigma}^2 \beta_0 dL_g + \int_{\partial\Sigma} \left(T_\eps\right)_{\vert \Sigma}^2 ( (\chi_0^\eps)^2 \beta_\eps - \beta_0) dL_g,    \\  \end{align*}
where by definition of $\beta_\eps$,
\begin{align*} \vert (\chi_0^\eps)^2 \beta_\eps - \beta_0 \vert = & \left\vert \left((\chi_0^\eps)^2-1\right)\beta_0 + \left(\chi_0^\eps\right)^2  \sum_{i=1}^k  \eta(z^i)  e^{-u_i(z^i)} \beta_i\left( \zeta^{-1}\left(  \frac{z^i}{\eps}\right) \right) \frac{2 \eps}{(\eps + z^i_2)^2 + \left(z^i_1\right)^2} \right\vert \\
\leq & C \left( \sum_i \mathbf{1}_{ \{ \vert z^i \vert \leq \eps^{\frac{1}{8}} \} } + \eps^{\frac{1}{2}}  \right) \end{align*}
implies that 
$$ \int_{\partial\Sigma} (\psi_0^\eps)^2 \beta_\eps dL_g = \int_{\partial \Sigma} \left(T_\eps\right)_{\vert \Sigma}^2 \beta_0 dL_g + O(\eps^{\frac{1}{8}}) $$
as $\eps \to 0$. We finally compute for $1\leq i \leq k$,
\begin{align*} & \int_{\mathbb{R}\times \{0\}} \psi_i^\eps( \zeta^{-1}(s))^2 \beta_\eps( \theta_i^{-1}( \eps  s  ) ) \eps ds = \int_{\mathbb{S}^1} \left(T_\eps\right)_{\vert D_i}^2 \beta_i dL_{\xi} \\ 
& + \int_{\mathbb{R}\times \{0\}} \left(\left(T_\eps\right)_{\vert D_i}\left( \zeta^{-1}(s)\right) \right)^2 \left( \eta_{\eps^{\frac{1}{4}}}\left( \eps s  \right)  \beta_\eps( \theta_i^{-1}( \eps  s  ) ) \eps - \beta_i(\zeta^{-1}(s)) \frac{2}{1+s_1^2} \right) ds \end{align*}
where by definition of $\beta_\eps$,
\begin{align*}
& \left\vert \eta_{\eps^{\frac{1}{4}}}\left( \eps s  \right)  \beta_\eps( \theta_i^{-1}( \eps  s  ) ) \eps - \beta_i(\zeta^{-1}(s)) \frac{2}{1+s_1^2} \right\vert \\
\leq & \left\vert ( \eta_{\eps^{\frac{1}{4}}}\left( \eps s  \right) -1 )\beta_i(\zeta^{-1}(s)) \frac{2}{1+s_1^2} \right\vert + \left\vert \eta_{\eps^{\frac{1}{4}}}\left( \eps s  \right) \beta_0( \theta_i^{-1}(\eps s) )  e^{u_i(\eps s )} \eps \right\vert \\
\leq & C \left( \frac{\mathbf{1}_{\vert s \vert \geq \eps^{- \frac{1}{2}}}}{1+s_1^2} + \eps \mathbf{1}_{\vert s \vert \leq \eps^{-\frac{3}{4}}} \right)
\end{align*}
implies that
$$  \int_{\mathbb{R}\times \{0\}} \psi_i^\eps( \zeta^{-1}(s))^2 \beta_\eps( \theta_i^{-1}( \eps  s  ) ) \eps ds = \int_{\mathbb{S}^1} \left(T_\eps\right)_{\vert D_i}^2 \beta_i dL_{\xi} + O(\eps^{\frac{1}{4}}) $$
as $\eps \to 0$. The previous estimates lead to
\begin{equation} \label{eq:eql2psieps} \int_{\partial \Sigma}  \psi_\eps ^2 \beta_\eps dL_g = \int_{\partial \Sigma} \left(T_\eps\right)_{\vert \Sigma}^2 \beta_0 dL_g + \sum_{i=1}^k \int_{\mathbb{S}^1} \left(T_\eps\right)_{\vert D_i}^2 \beta_i dL_{\xi} +  O\left( \eps^{\frac{1}{8}}\right)\end{equation}
as $\eps \to 0$. Now, gathering \eqref{eq:eqtestpsieps} \eqref{eq:eqgradpsieps} and \eqref{eq:eql2psieps}, we obtain 
\begin{equation} \label{eq:eqtestTeps} \sigma_m(\Sigma,g, \beta_\eps) \leq \frac{\int_{\tilde{\Sigma}} \left\vert \nabla T_\eps \right\vert_{\tilde{g}}^2 dA_{\tilde{g}} + O\left( \left(\ln \frac{1}{\eps}\right)^{-1}\right)}{\int_{\partial \tilde{\Sigma}} T_\eps^2 \tilde{\beta} dL_{\tilde{g}} +
 O(\eps^{\frac{1}{8}}) }  \end{equation}
 where 
 \begin{align*}
\frac{\int_{\tilde{\Sigma}} \left\vert \nabla T_\eps \right\vert_{\tilde{g}}^2 dA_{\tilde{g}} }{\int_{\partial \tilde{\Sigma}} T_\eps^2 \tilde{\beta} dL_{\tilde{g}} } = 
 \frac{\sum_{j=0}^m \left(a_j^\eps\right)^2 \int_{\tilde{\Sigma}} \vert \nabla \varphi_j \vert^2_g dA_g }{\sum_{j=0}^m \left(a_j^\eps\right)^2 \int_{\partial \tilde{\Sigma}} \left(\varphi_j\right)^2 \tilde{\beta} dL_{\tilde{g}} } = \sum_{j=0}^m \left(a_j^\eps\right)^2 \sigma_j(\tilde{\Sigma}, \tilde{g},\tilde{\beta}) \leq \sigma_m(\tilde{\Sigma}, \tilde{g},\tilde{\beta}).
 \end{align*}
 Finally, using \eqref{eq:eqtestTeps}, the previous inequality and that
 $$ \int_{\Sigma} \beta_\eps dL_g = \int_{\tilde{\Sigma}} \tilde{\beta} dL_{\tilde{g}} + O(\eps^{\frac{1}{8}}) $$
as $\eps\to 0$, we obtain the Claim. 
\end{proof}

We now denote for $j \in \mathbb{N}^*$, $\sigma_j^\eps = \sigma_j(\Sigma,g,\beta_\eps)$ and $\varphi_j^\eps$ functions such that
\begin{equation} \label{eq:varphijeps}
\begin{cases}
\Delta_g \varphi_j^\eps = 0  \text{ in } \Sigma \\
\partial_\nu \varphi_j^\eps = \sigma_j^\eps \beta_\eps \varphi_j^\eps \text{ on } \partial \Sigma,
\end{cases}
\end{equation}
and we assume for $j,j' \in \mathbb{N}^*$ that
\begin{equation}\label{eq:orthonormalbasis}\int_{\partial \Sigma} \varphi_j^\eps \varphi_{j'}^\eps \beta_\eps dL_g = \delta_{j,j'}.\end{equation}

For $1 \leq i \leq k$, and $s \in \mathbb{R}^2_+$ and $\eps$ small enough, we denote by 
$$ \left(\varphi_j^\eps\right)_i(s) = \varphi_j^\eps\left( \theta_i^{-1}( \eps s)  \right) $$
and for $s \in \mathbb{R} \times \{0\}$,
$$ \left(\beta_\eps\right)_i(s) = \frac{2 \beta_i(\zeta^{-1}(s))}{1+s_1^2} + \beta_0\left( \theta_i^{-1}( \eps s)  \right)  e^{u_i(\eps s )} \eps  $$
and we obtain the equations for any $R>0$ and $\eps$ small enough:
\begin{equation} \label{eq:varphijepsi}
\begin{cases}
\Delta \left(\varphi_j^\eps\right)_i = 0  \text{ in } \mathbb{D}_R^+  \\
\partial_r \left(\varphi_j^\eps\right)_i = \sigma_j^\eps (\beta_\eps)_i \left(\varphi_j^\eps\right)_i \text{ on } [-R,R] \times \{0\}.
\end{cases}
\end{equation}
Notice that for any $R>0$, $(\beta_\eps)_i$ is uniformly bounded and uniformly lower bounded in $[-R,R]\times \{0\}$ as $\eps \to 0$ because
\begin{equation} \label{eq:convbetaepsi} (\beta_\eps)_i \to \frac{2 \beta_i\circ \zeta^{-1}}{1+s_1^2} \text{ in } \mathcal{C}^2([-R,R]\times \{0\})\end{equation}
as $\eps \to 0$. In particular,
$$ \int_{[-R,R]\times \{0\}} ((\varphi_j^\eps)_i)^2 ds \leq C\int_{[-R,R]\times \{0\}} ((\varphi_j^\eps)_i)^2 (\beta_\eps)_i ds \leq \int_{\partial \Sigma} (\varphi_j^\eps)^2 \beta_\eps dL_g = 1. $$
Using this inequality, that $(\sigma_\eps^j)$ is bounded as $\eps \to 0$ (see Claim \ref{cl:boundedsigma}) and that $(\beta_\eps)_i$ is uniformly bounded in $[-R,R]\times \{0\}$ as $\eps \to 0$, 
standard elliptic theory on the equation \eqref{eq:varphijepsi} implies up to the extraction of a subsequence, the existence of $(\varphi_j)_i : \mathbb{R}^2_+ \to \R$  such that for any $R >0$
\begin{equation}\label{eq:convvarphiepsi} (\varphi_j^\eps)_i \to (\varphi_j)_i \text{ in } \mathcal{C}^2(\mathbb{D}_R^+) \end{equation}
as $\eps \to 0$. Passing to the limit as $\eps\to 0$ in the equation \eqref{eq:varphijepsi} implies
\begin{equation} \label{eq:varphijepsilim}
\begin{cases}
\Delta_\xi \left(\varphi_j\right)_i = 0  \text{ in } \mathbb{R}^2_+  \\
-\partial_{s_2} \left(\varphi_j\right)_i = \sigma_j \frac{2\beta_i \circ \zeta^{-1}}{1+s_1^2} \left(\varphi_j\right)_i \text{ on } \mathbb{R} \times \{0\},
\end{cases}
\end{equation}
where up to the extraction of a subsequence, we let $\sigma_j^\eps \to \sigma_j$ as $\eps \to 0$. Notice that by  Claim \ref{cl:boundedsigma}, $\sigma_j \leq \sigma_j(\tilde{\Sigma},\tilde{g},\tilde{\beta})$. Applying the biholomorphism $\zeta$, we obtain from \eqref{eq:varphijepsilim} an equation on $\left(\varphi_j\right)_i \circ \zeta$ that holds in $\mathbb{D} \setminus \{-1\}$. However, knowing that $\phi_j^i = \left(\varphi_j\right)_i \circ \zeta$ is bounded in $H^1( \mathbb{D} )$, the equation can be extended in $\mathbb{D}$ and we obtain:
\begin{equation} \label{eq:varphijepsilimdisk}
\begin{cases}
\Delta_\xi \phi_j^i = 0  \text{ in } \mathbb{D}  \\
\partial_r \phi_j^i = \sigma_j \beta_i  \phi_j^i \text{ on } \mathbb{S}^1,
\end{cases}
\end{equation}
so that $\phi_j^i$ is an eigenfunction for $(\mathbb{D},\xi,\beta_i)$.

Similarly, setting 
$$  \Sigma_R = \Sigma \setminus \bigsqcup_{i=1}^k \theta_i^{-1}( \mathbb{D}_{R^{-1}}^+ ) \text{ and }  I_R = \partial \Sigma \setminus \bigsqcup_{i=1}^k \theta_i^{-1}( [-R^{-1},R^{-1}]\times \{0\} ), $$
we have that for any $R>0$, $(\beta_\eps)$ is uniformly bounded in $I_R$ as $\eps \to 0$ because
\begin{equation} \label{eq:convbetaeps} \beta_\eps \to \beta_0 \text{ in } \mathcal{C}^2\left(I_R \right)\end{equation}
as $\eps \to 0$. In particular, if $\beta_0 \neq 0$,
$$ \int_{I_R} (\varphi_j^\eps)^2 ds \leq C\int_{I_R} (\varphi_j^\eps)^2 \beta_\eps dL_g \leq \int_{\partial \Sigma} (\varphi_j^\eps)^2 \beta_\eps dL_g = 1. $$
Assuming that $\beta_0 \neq 0$, using this inequality, that $(\sigma_\eps^j)$ is bounded as $\eps \to 0$ (see Claim \ref{cl:boundedsigma}) and that $(\beta_\eps)$ is uniformly bounded in $I_R$ as $\eps \to 0$, standard elliptic theory on the equation \eqref{eq:varphijeps} implies up to the extraction of a subsequence, the existence of $\varphi_j : \Sigma \setminus \{x_1,\cdots,x_k\} \to \R$  such that for any $R >0$
\begin{equation} \label{eq:convvarphieps} \varphi_j^\eps \to \varphi_j \text{ in } \mathcal{C}^2(\Sigma_R) \end{equation}
as $\eps \to 0$. Passing to the limit as $\eps\to 0$ in the equation \eqref{eq:varphijeps} implies
\begin{equation} \label{eq:varphijepslim}
\begin{cases}
\Delta_g \varphi_j = 0  \text{ in } \Sigma \setminus \{x_1,\cdots, x_k\}  \\
\partial_\nu \varphi_j = \sigma_j \beta_0 \varphi_j\text{ on } \partial\Sigma \setminus \{x_1,\cdots, x_k\}.
\end{cases}
\end{equation}
Again, since $\varphi_j$ is bounded in $H^1(\Sigma)$, we can extend the equation \eqref{eq:varphijepslim} to $\Sigma$ and we obtain a function $\phi_j^0 = \varphi_j $
\begin{equation} \label{eq:varphijepslimsigma}
\begin{cases}
\Delta_g \phi_j^0 = 0  \text{ in } \Sigma   \\
\partial_\nu \phi_j^0 = \sigma_j \beta_0 \phi_j^0 \text{ on } \partial\Sigma.
\end{cases}
\end{equation}

We prove now that \eqref{eq:orthonormalbasis} passes to the limit as $\eps \to 0$:
\begin{cl} \label{cl:orthonormal}
We have for $j,j' \in \mathbb{N}$:
$$  \int_{\partial \Sigma} \phi_j^0 \phi_{j'}^0 \beta_0 dL_g + \sum_{i=1}^k \int_{\mathbb{S}^1} \phi_j^i \phi_{j'}^i \beta_i d\theta = \delta_{j,{j'}},$$
where by convention the first term is $0$ if $\beta_0 =0$.
\end{cl}

\begin{proof} We first show that it suffices to prove that no mass accumulates in the neck part:
\begin{equation} \label{eq:nomassneck} \lim_{R \to +\infty} \limsup_{\eps\to 0} \int_{ N_{R,\eps} }   (\varphi_j^\eps)^2 \beta_\eps dL_g = 0. \end{equation}
where we denote
$$ N_{R,\eps} = \bigsqcup_{i=1}^k N_{R,\eps}^i \text{ and } N_{R,\eps}^i = \theta_i^{-1}( \left([-R^{-1},R^{-1}] \setminus [-R\eps^{-1},R\eps^{-1}]\right) \times \{0\}  ). $$
and in addition, we have that for any $R>0$,
\begin{equation} \label{eq:nomassthinpart} \beta_0 = 0 \Rightarrow \limsup_{\eps \to 0} \int_{I_R} (\varphi_j^\eps)^2 \beta_\eps dL_g \to 0  \end{equation}
as $\eps \to 0$. Indeed, we have by \eqref{eq:convbetaeps} and \eqref{eq:convvarphieps} that
$$ \beta_0 \neq 0 \Rightarrow \int_{I_R} \varphi_j^\eps \varphi_{j'}^\eps \beta_\eps dL_g \to \int_{I_R} \phi_j^0 \phi_{j'}^0  \beta_0 dL_g $$
as $\eps \to 0$ and by \eqref{eq:convbetaepsi} and \eqref{eq:convvarphiepsi} that
\begin{equation*}\begin{split}& \int_{\partial \Sigma \setminus I_{\eps R^{-1}}} \varphi_j^\eps \varphi_{j'}^\eps \beta_\eps dL_g = \sum_{i=1}^k \int_{[-R,R]\times \{0\}} \left(\varphi_j^\eps\right)_i \left(\varphi_{j'}^\eps\right)_i \left(\beta_\eps\right)_i d\theta \\ & \to  \sum_{i=1}^k \int_{[-R,R]\times \{0\}} \left(\varphi_j \right)_i \left(\varphi_{j'} \right)_i \frac{2\beta_i \circ \zeta^{-1}}{1+s_1^2} ds = \int_{\zeta^{-1}([-R,R]\times \{0\})} \phi_j^i \phi_{j'}^i \beta_i d\theta\end{split} \end{equation*}
as $\eps \to 0$. Therefore, if \eqref{eq:nomassneck} and \eqref{eq:nomassthinpart} hold, using all the previous convergences, using \eqref{eq:orthonormalbasis}, and passing to the limit as $R \to +\infty$ leads to
$$ \left\vert \int_{\partial \Sigma} \phi_j^0 \phi_{j'}^0 \beta_0 dL_g + \sum_{i=1}^k \int_{\mathbb{S}^1} \phi_j^i \phi_{j'}^i \beta_i d\theta - \delta_{j,{j'}}  \right\vert \leq \lim_{R\to +\infty} \limsup_{\eps \to 0}\int_{N_{R,\eps}}  (\varphi_j^\eps)^2 \beta_\eps dL_g =0 $$
if $\beta_0 \neq 0$ and
$$ \left\vert \sum_{i=1}^k \int_{\mathbb{S}^1} \phi_j^i \phi_{j'}^i \beta_i d\theta - \delta_{j,{j'}}  \right\vert \leq \lim_{R\to +\infty} \limsup_{\eps \to 0}\int_{N_{R,\eps} \cup I_R}  (\varphi_j^\eps)^2 \beta_\eps dL_g = 0 $$
if $\beta_0 = 0$. 

It remains to prove \eqref{eq:nomassneck} and \eqref{eq:nomassthinpart}. We will need several steps:

\medskip

\noindent\textbf{Step 1:} There is a constant $C>0$ such that
\begin{equation} \label{eq:lnestneck} \forall z \in \mathbb{D}^+_{\frac{\delta}{\eps}}, \vert (\varphi_j^\eps)_i(z) \vert \leq C  \sqrt{\ln\left( 2+ \vert z \vert\right)}.\end{equation}
Moreover, we have that
\begin{equation} \label{eq:coarseuniformbound} \forall x \in \Sigma, \vert \varphi_\eps^j(x) \vert \leq C \sqrt{\ln \frac{1}{\eps}} . \end{equation}

\medskip

\noindent\textbf{Proof of Step 1:} We first prove \eqref{eq:lnestneck}. We denote $ f_\eps = (\varphi_j^\eps)_i$. Let $1 \leq R_\eps \leq \delta \eps^{-1}$. The following computation on mean values $\bar{f}_\eps(r) = \frac{1}{\pi} \int_{0}^{\pi} f_\eps(r e^{i\theta})d\theta $ leads to
\begin{align*} &  \bar{ f}_\eps(R_\eps) -\bar{f}_\eps(1) = \frac{1}{\pi} \int_{1}^{R_\eps} \int_{0}^{\pi} \partial_r f_\eps(re^{i\theta})d\theta 
= \frac{1}{\pi} \int_{1}^{R_\eps} \int_0^{\pi} \partial_r (\ln r)\partial_r f_\eps(re^{i\theta}) rdrd\theta 
\\&  = \frac{1}{\pi} \int_{\mathbb{D}_{R_\eps}^+ \setminus \mathbb{D}^+} \langle\nabla \ln \vert x \vert \nabla f_\eps \rangle_\xi dA_\xi \leq \frac{1}{\sqrt{\pi}} \left(\int_1^{R_\eps} \frac{dr}{r}\right)^{\frac{1}{2}} \left(\int_{\Sigma} \vert \nabla \varphi_j^\eps \vert_g^2 dA_g\right)^{\frac{1}{2}} \leq \frac{ \sqrt{\ln R_\eps}}{\sqrt{\pi}} \sqrt{ \sigma_j^\eps}.  \end{align*}
Now, for  $1 \leq r \leq  \frac{\delta}{4} \eps^{-1}$, we set $\tilde{f}_\eps(x) = f_\eps(r x) - \bar{f}_\eps(r)$ and it satisfies the system of equation:
\begin{equation} \label{eq:sysE4}
\begin{cases}
\Delta \tilde{f}_\eps = 0 \text{ in } E_4  \\
- \partial_{2} \tilde{f}_\eps(s) = \sigma_\eps^j V_\eps(s) \tilde{f}_\eps(s) + \sigma_\eps^j  V_\eps(s)  \bar{f}_\eps(r) \text{ for } s \in E_4 \cap \mathbb{R}^2 \times \{0\}.
\end{cases}
\end{equation}
where 
$$ E_\rho = \begin{cases} \mathbb{D}_\rho \setminus \mathbb{D}_{\frac{1}{\rho}} \text{ if } r \geq 2 \\
\mathbb{D}_\rho \text{ if } r \leq 2, \end{cases} $$
and
$$ V_\eps(s) = r(\beta_\eps)_i(r s ) = \frac{2r}{1+r^2 s_1^2} + \eps r \beta_0(\theta_i^{-1}(\eps r s) )e^{u_i(\eps r s)} $$
is uniformly bounded by a positive constant independent of $\eps$ and $r$. By standard elliptic theory, we obtain that
$$ \Vert \tilde{f}_\eps \Vert_{L^\infty(E_2)} \leq C \sigma_\eps^j \bar{f}_\eps(r). $$
By \eqref{eq:convvarphiepsi}, $\bar{f}_\eps(1)$ is uniformly bounded. Then, gathering all the previous estimates, we obtain \eqref{eq:lnestneck}.

Now, we prove \eqref{eq:coarseuniformbound}. From \eqref{eq:lnestneck}, we deduce that \eqref{eq:coarseuniformbound} holds true in $ \Sigma \setminus \Sigma_{2 \delta^{-1}}$. It is also clear from \eqref{eq:convvarphieps} that if $\beta_0 \neq 0$, \eqref{eq:coarseuniformbound} holds true in $\Sigma_{4 \delta^{-1}}$. It remains to prove \eqref{eq:coarseuniformbound} in $\Sigma_{4 \delta^{-1}}$ if $\beta_0 = 0$. Now, we use the equation
$$ \begin{cases}
\Delta \varphi_\eps^j = 0 \\
\partial_\nu \varphi_\eps^j = \sigma_j^\eps \beta_\eps \varphi_\eps^j
\end{cases} $$
where we know that 
$$\Vert \beta_\eps \Vert_{L^{\infty}(\Sigma_{4 \delta^{-1}})} \leq O(\eps) $$
as $\eps \to 0$. Then, by standard elliptic theory, knowing that $\vert \varphi_\eps^j \vert \leq O(\sqrt{\ln \frac{1}{\eps}})$ as $\eps \to 0$, in $\Sigma_{4 \delta^{-1}} \setminus \Sigma_{2 \delta^{-1}}$, we obtain that 
$$\Vert \varphi_\eps^j \Vert_{L^{\infty}(\Sigma_{4 \delta^{-1}})} \leq C\sqrt{\ln \frac{1}{\eps}} + C \sigma_j^\eps \Vert \beta_\eps \Vert_{L^{\infty}(\Sigma_{4 \delta^{-1}})} \Vert \varphi_\eps^j \Vert_{L^{\infty}(\Sigma_{4 \delta^{-1}})} \leq C \sqrt{\ln \frac{1}{\eps}} + C \eps \Vert \varphi_\eps^j \Vert_{L^{\infty}(\Sigma_{4 \delta^{-1}})}  $$
and substracting the last term completes the proof ot \eqref{eq:coarseuniformbound}.

\medskip

\noindent \textbf{Step 2:} If $\beta_0 \neq 0$, then there is a constant $C>0$ such that for any $ R > \frac{2}{\delta}$,
\begin{equation} \label{eq:lnestneck2} \forall x \in N_{R,\eps}^i, \vert \varphi_\eps^j(x) \vert \leq C \vert \ln \vert z^i \vert \vert. \end{equation}

\medskip

\noindent\textbf{Proof of Step 2:} We proceed exactly as in the proof of \eqref{eq:lnestneck}. We denote $h_\eps = \varphi_\eps^j \circ \theta_i^{-1}$, and $\bar{h}_\eps = \frac{1}{\pi} \int_0^\pi h_\eps(re^{i\theta})d\theta$. We obtain for $0 < r_\eps \leq \delta$ that
\begin{align*} &  \bar{ h}_\eps(1) -\bar{h}_\eps(r_\eps) = \frac{1}{\pi} \int_{r_\eps}^{\delta} \int_{0}^{\pi} \partial_r h_\eps(re^{i\theta})d\theta 
= \frac{1}{\pi} \int_{r_\eps}^{\delta} \int_0^{\pi} \partial_r (\ln r)\partial_r f_\eps(re^{i\theta}) rdrd\theta 
\\&  = \frac{1}{\pi} \int_{\mathbb{D}_{\delta}^+ \setminus \mathbb{D}_{r_\eps}^+} \langle\nabla \ln \vert x \vert \nabla f_\eps \rangle_\xi dA_\xi \leq \frac{1}{\sqrt{\pi}} \left(\int_{r_\eps}^{\delta} \frac{dr}{r}\right)^{\frac{1}{2}} \left(\int_{\Sigma} \vert \nabla \varphi_j^\eps \vert_g^2 dA_g\right)^{\frac{1}{2}} \leq \frac{ \sqrt{\ln\frac{1}{r_\eps}}}{\sqrt{\pi}} \sqrt{ \sigma_j^\eps}.  \end{align*}
We notice that $f_\eps( x) = h_\eps(\eps x) $, we use again \eqref{eq:sysE4} to obtain
$$ \Vert \tilde{f}_\eps \Vert_{L^\infty(E_2)} \leq C \sigma_\eps^j \bar{f}_\eps(r) $$
and by \eqref{eq:convvarphieps}, $\bar{h}_\eps(\delta)$ is uniformly bounded. Then, gathering all the previous estimates, we obtain \eqref{eq:lnestneck2}.

\medskip

\noindent \textbf{Step 3:} We prove \eqref{eq:nomassneck} if $\beta_0 \neq 0$. We set 
$$ \tilde{N}_{R,\eps} = \left([-(R\eps)^{-1},(R\eps)^{-1}]\setminus [-R,R]\right) \times \{0\}$$
$$ \int_{N_{R,\eps}^i} (\varphi_j^\eps)^2 \beta_\eps dL_g = \int_{\tilde{N}_{R,\eps}} ((\varphi_j^\eps)_i)^2 \beta_i(\zeta^{-1}(s)) \frac{2}{1+s_1^2} ds  + \int_{N_{R,\eps}^i} (\varphi_j^\eps)^2 \beta_0 dL_g $$
and we have from \eqref{eq:lnestneck} that
$$ \limsup_{\eps \to 0} \int_{\tilde{N}_{R,\eps}} ((\varphi_j^\eps)_i)^2 \frac{2}{1+s_1^2} ds \leq \limsup_{\eps \to 0}  C  \int_R^{(R\eps)^{-1}} \frac{\ln(2+\vert u \vert)}{1+u^2} = C  \int_R^{+\infty} \frac{\ln(2+\vert u \vert)}{1+u^2}.$$
and from\eqref{eq:lnestneck2} that 
$$ \limsup_{\eps \to 0} \int_{N_{R,\eps}^i} (\varphi_j^\eps)^2 \beta_0 dL_g \leq  \limsup_{\eps \to 0} \int_{\eps R}^{R^{-1}} \vert \ln\left(\vert u \vert \right) \vert du = \int_{0}^{R^{-1}} \vert \ln(\vert u \vert) \vert du. $$
Letting $R\to +\infty$ in all the previous inequalities completes the proof of Step 3.

\medskip

\noindent \textbf{Step 4}: We assume now that $\beta_0 = 0$ and we prove \eqref{eq:nomassneck} and \eqref{eq:nomassthinpart}. As previously, with $\beta_0 = 0$, we easily check from
$$ \int_{N_{R,\eps}^i} (\varphi_j^\eps)^2 \beta_\eps dL_g = \int_{\tilde{N}_{R,\eps}} ((\varphi_j^\eps)_i)^2 \frac{2}{1+s_1^2} ds  $$
that Step 1 leads to:
$$ \lim_{R\to +\infty} \limsup_{\eps \to 0} \int_{\tilde{N}_{R,\eps}} ((\varphi_j^\eps)_i)^2 \frac{2}{1+s_1^2} ds = 0,$$
and we proved \eqref{eq:nomassneck}. Now, let's prove \eqref{eq:nomassthinpart}. By definition of $\beta_\eps$ and by \eqref{eq:coarseuniformbound},
$$ \left\vert \int_{I_R} (\varphi_j^\eps)^2 \beta_\eps dL_g \right\vert \leq C \sum_{i=1}^k \int_{\frac{1}{R}}^\delta \left(\ln\frac{1}{\eps}\right) \frac{2\eps }{\eps^2 + s^2} ds \to 0   $$
as $\eps\to 0$. Letting $R \to +\infty$ completes the proof of \eqref{eq:nomassthinpart}.

\end{proof}

Now we recall that $\sigma_j$ is (up to the extraction of a subsequence) the limit of $\sigma_j^\eps = \sigma_j(\Sigma,g,\beta_\eps)$ as $\eps\to 0$ and that by Claim \ref{cl:boundedsigma}, we have that $\sigma_j \leq \sigma_j(\tilde{\Sigma},\tilde{g},\tilde{\beta})$. It remains to prove the converse inequality to obtain Proposition \ref{prop:convmetrictest}. We let $\phi_j : \tilde{\Sigma} \to \R$ be the function defined as $(\phi_j)_{\vert \Sigma} = \phi_j^0$ in $\Sigma$ (if $\Sigma$ appears in the disjoint union of connected surfaces of $\tilde{\Sigma}$) and $(\phi_j)_{\vert D_i} = \phi_j^i$ for $1 \leq i \leq k$. Testing $\phi_0,\cdots,\phi_m$ in the variational characterization of $\sigma_m(\tilde{\Sigma},\tilde{g},\tilde{\beta})$ yields
$$ \sigma_m(\tilde{\Sigma},\tilde{g},\tilde{\beta}) \leq \max_{a \in \mathbb{S}^m} \frac{\int_{\tilde{\Sigma}} \vert \nabla \sum_{j} a_j \phi_j \vert_{\tilde{g}}^2 dA_{\tilde{g}}  }{\int_{\partial \tilde{\Sigma}} \left( \sum_{j} a_j \phi_j \right)^2 \tilde{\beta} dL_{\tilde{g}}  } \leq \sum_j a_j^2 \sigma_j \leq \sigma_m,$$
where we used \eqref{eq:varphijepsilimdisk}, \eqref{eq:varphijepslimsigma} and Claim \ref{cl:orthonormal}. The proof of Proposition \ref{prop:convmetrictest} is now complete.

\section{Proof of Proposition \ref{prop:construction>2pi}} \label{sec3}
The proof of Proposition \ref{prop:construction>2pi} is based on techniques used in \cite{PetridesVarMethod,PetridesNew}. We will refer to Propositions and lemmas of these papers all along the proof. Let's first define the family of surfaces $\Sigma_{l,\eps}^{\pm}$ we work with.
We denote the rectangle
$$R_{l,\eps} = \left[-\frac{\eps l}{2},\frac{\eps l}{2}\right]\times\left[-\frac{\eps}{2},\frac{\eps}{2}\right]$$
 and its boundary components
$$ I_{l,\eps} =  \left[-\frac{\eps l}{2},\frac{\eps l}{2}\right]\times \left\{-\frac{\eps}{2},\frac{\eps}{2}\right\} \text{ and }  J_{l,\eps} =  J_{l,\eps}^{+}\cup  J_{l,\eps}^{-} \text{ where }  J_{l,\eps} = \left\{ \pm \frac{l \eps}{2} \right\} \times \left[-\frac{\eps}{2},\frac{\eps}{2}\right]. $$ 
We glue $R_{l,\eps}$  to a disk $\mathbb{D}$ via the attaching map
$$ \theta_{\pm} : (\eta \frac{l\eps}{2}, t) \in J_{l,\eps}  \mapsto  \eta e^{ \pm i \eta t} $$
and we denote by 
\begin{equation}\label{def:sigmalepspm} \Sigma_{l,\eps}^{\pm} = (\mathbb{D} \sqcup R_{l,\eps})/\sim_{\theta_{\pm}} \end{equation}
the resulting surface endowed with the metric $g_\eps$ that equals to the flat metric in $R_{l,\eps}$ and the flat metric in $\mathbb{D}$. Notice that
$$ \partial \Sigma_{l,\eps}^\pm = (\mathbb{S}^1 \setminus J_{l,\eps}) \cup I_{l,\eps} $$
and that $\Sigma_{l,\eps}^{\pm}$ is homeomorphic to an annulus if $\theta_{\pm} = \theta_+$ and a Möbius band if $\theta_{\pm} = \theta_-$.

\begin{proof}[Proof of Proposition \ref{prop:construction>2pi}] We argue by contradiction. We assume that there is $l>0$ such that up to the extraction of a subsequence of $\eps \to 0$, $\sigma_1(\Sigma_{{l},\eps}^{\pm},[g_\eps] ) = 2\pi$. Notice that all along the proof, the subsequences of $\eps \to 0$ that are extracted are not written explicitely. For instance "as $\eps \to 0$" means "as the subsequence of $\eps \to 0$ we consider goes to $0$".

By \cite[Claim 3.2]{PetridesNew},
\begin{equation}\label{eq:convsigmadeltaepssquare} \bar{\sigma}_k(\Sigma_{{l},\eps}^{\pm},g_\eps  ) \geq \bar{\sigma}_k(\mathbb{D}) - (\delta_{\eps,l})^2 \end{equation}
where $\bar{\sigma}_k(\mathbb{D}) = 2\pi k$ and $\delta_{\eps,l} = c_l \eps^{\frac{1}{2}}\left(\ln \frac{1}{\eps}\right)^{\frac{1}{4}}$ for some constant $c_l >0$. Notice in particular that for a given $\eps >0$,
\begin{equation}\label{eq:convsigmadeltaepssquare1} \bar{\sigma}_1(\Sigma_{{l},\eps}^{\pm},g_\eps) \geq \sigma_1(\Sigma_{{l},\eps}^{\pm},[g_\eps]) - (\delta_{\eps,l})^2 \end{equation}
and we aim at appling Ekeland's variational principle to this estimate. We set
$$ \mathcal{A}_\eps = \{ \beta \in \overline{X} ; \beta(1,1) \geq 2\pi \},$$
where $\overline{X}$ is the closure of the set 
$$ \left\{ (\phi,\psi) \mapsto \int_{\partial \Sigma_{l,\eps}^\pm} \phi\psi e^{u} dL_{g_\eps}   ; u \in \mathcal{C}^\infty(\partial\Sigma) \right\} $$
in the Banach space of continuous bilinear forms on $H^1(\Sigma_{l,\eps}^\pm)$ denoted by $B(\Sigma_{l,\eps}^\pm)$. We endow $\mathcal{A}_\eps$ with the distance $d_{g_\eps}$ induced by the norm
$$ \Vert \beta \Vert_{g_\eps} = \sup_{\phi,\psi \in H^1(\Sigma_{l,\eps}^\pm)} \frac{\beta(\phi,\psi)}{\Vert \phi \Vert_{H^1,g_\eps}\Vert \psi \Vert_{H^1,g_\eps}} $$
on $B(\Sigma_{l,\eps}^\pm)$, where
$$ \Vert \phi \Vert_{H^1,g_\eps}^2 = \int_{\Sigma_{l,\eps}^\pm } \vert \nabla \phi \vert^2_{g_\eps} dA_{g_\eps} + \int_{\partial \Sigma_{l,\eps}^\pm } \vert \nabla \phi \vert^2_{g_\eps} dL_{g_\eps}.$$
Knowing \eqref{eq:convsigmadeltaepssquare1}, and upper semi-ontinuity of eigenvalues (see \cite[Proposition 1.1]{PetridesVarMethod}), Ekeland's variational principle on $(\mathcal{A}_\eps,d_{g_\eps})$ gives $\beta_\eps \in \mathcal{A}_\eps$ such that
\begin{equation} \label{eq:ekeland} \begin{cases} \bar{\sigma}_1(\Sigma_{l,\eps}^{\pm},g_\eps,\beta_\eps) \geq \bar{\sigma}_1(\Sigma_{l,\eps}^{\pm},g_\eps),\\
 d_{g_\eps}(\beta_\eps, dL_{g_\eps}) \leq \delta_{\eps,l}, \\
  \forall \beta \in \overline{X}, \bar{\sigma}_1(\Sigma_{l,\eps}^{\pm},g_\eps,\beta_\eps) - \bar{\sigma}_1(\Sigma_{l,\eps}^{\pm},g_\eps,\beta) \geq - \delta_{\eps,l} d_{g_\eps}(\beta_\eps, \beta). \end{cases}\end{equation}
Notice that
$$ \vert \beta_\eps(1,1) - L_{g_\eps}(\partial\Sigma_{l,\eps}^{\pm}) \vert \leq \delta_{l,\eps}$$
so that
\begin{equation} \label{eq:betaeps11} \beta_\eps(1,1) = 2\pi + O(\delta_{\eps,l}) \end{equation}
as $\eps \to 0$. Another consequence coming from \cite[Proposition 1.1]{PetridesNew} is that for any $k$,
\begin{equation}\label{eq:lip} \vert \sigma_k(\Sigma_{l,\eps}^\pm,g_\eps,\beta_\eps) - \sigma_k(\Sigma_{l,\eps}^\pm,g_\eps) \vert \leq O(\delta_{\eps,l}) \end{equation}
as $\eps \to 0$.

From \cite[Proposition 1.5]{PetridesVarMethod}, we obtain the existence of $\Phi_{\eps} : \Sigma_{l,\eps}^{\pm} \to \mathbb{R}^{n_\eps}$ such that denoting $\sigma_\eps = \sigma_1(\Sigma_{l,\eps}^\pm,g_\eps,\beta_\eps)$,
\begin{equation} \label{eq:eulerlagrange} \begin{cases} \Delta_{g_\eps} \Phi_\eps = \sigma_\eps \beta_\eps(\Phi_\eps,\cdot) \\ \beta_\eps(\Phi_\eps,\Phi_\eps) = \beta_\eps(1,1) \\
 \vert \Phi_\eps \vert^2 \geq 1 - \theta_\eps^2 \text{ on } \partial \Sigma_{l,\eps}^{\pm}
\text{ where } \Delta_{g_\eps} \theta_\eps = 0 \text{ and }
 \Vert \theta_\eps \Vert_{H^1,g_\eps}^2 \leq \delta_{l,\eps}.
 \end{cases} \end{equation}
We deduce from \eqref{eq:lip} that 
$$\sigma_\eps =  \sigma_1(\Sigma_{l,\eps}^\pm,g_\eps,\beta_\eps) =  \sigma_2(\Sigma_{l,\eps}^\pm,g_\eps,\beta_\eps) \to 1$$
as $\eps\to 0$ and more generally that
$$ \sigma_k(\Sigma_{l,\eps}^\pm,g_\eps,\beta_\eps) \to k $$
as $\eps \to 0$.  In particular, since the coordinates of $\Phi_\eps$ are first eigenfunctions, we obtain that $n_\eps = 2$ for $\eps$ small enough. We let $\omega_\eps$ be the harmonic extension in $\Sigma_{l,\eps}^{\pm}$ of
$$ \omega_\eps = \sqrt{\theta_\eps^2 + \vert \Phi_\eps \vert^2 } \text{ on } \partial \Sigma_{l,\eps}^{\pm} $$
and we define $\Psi_\eps : (\Sigma_{l,\eps}^\pm,\partial \Sigma_{l,\eps}^\pm)  \to (\mathbb{B}^3,\mathbb{S}^2)$ as
$$ \Psi_\eps = \frac{(\Phi_\eps,\theta_\eps)}{\omega_\eps} = (\widetilde{\Phi}_\eps,\widetilde{\theta}_\eps).$$
That $\Psi_\eps$ maps into $\mathbb{B}^3$ is a consequence of the maximum principle, since it maps in $\mathbb{S}^2$ on $\partial \Sigma_{l,\eps}^\pm$ and $\Phi_\eps$, $\theta_\eps$ and $\omega_\eps$ are harmonic functions. Indeed, we can use the elliptic equation 
$$  -div_{g_\eps}\left( \omega_\eps^2 \nabla \Psi_\eps \right) = \omega_\eps \Delta_{g_\eps}(\Phi_\eps,\theta_\eps) - (\Phi_\eps,\theta_\eps) \Delta_{g_\eps} \omega_\eps = 0$$
to deduce that for any $X \in \mathbb{S}^2$,
$$ -div_{g_\eps}\left( \omega_\eps^2 \nabla \langle \Psi_\eps, X\rangle \right) = 0. $$
and $\langle \Psi_\eps, X\rangle$ cannot realize its maximum in the interior of $ \Sigma_{l,\eps}^\pm$ for any $X \in \mathbb{S}^2$.

Now, we have by \cite[Claim 3.1]{PetridesVarMethod} that
\begin{equation} \label{eq:clomegacloseto1} \beta_\eps(\omega_\eps,\omega_\eps) - \beta_\eps(1,1)+ \int_{\Sigma_{l,\eps}^{\pm}} \left( \vert \nabla \omega_\eps \vert^2_{g_\eps} + \vert \nabla \left(\widetilde{\Phi}_\eps - \Phi_\eps \right) \vert^2_{g_\eps} + (\omega_\eps^2-1) \vert \nabla \widetilde{\Phi}_\eps \vert^2_{g_\eps} \right)dA_{g_\eps} \leq O(\delta_{\eps,l}) \end{equation}
as $\eps \to 0$. From now on, the proof is self-contained but widely inspired from \cite{PetridesNew}.

\medskip

\noindent \textbf{Step 1:} We set 
$$ B(\psi,\psi) = \sum_{i=1}^2 \int_{\mathbb{D}} \vert \nabla \psi_i \vert^2_\xi dA_\xi - \int_{\mathbb{S}^1} \left(\psi_i-\left(\frac{1}{2\pi} \int_{\mathbb{S}^1} \psi_i d\theta \right)\right)^2 d\theta. $$ 
and we prove
\begin{equation} \label{eq:estmeanvalue}
\frac{1}{2\pi}\left\vert \int_{\mathbb{S}^1} \widetilde{\Phi}_\eps d\theta \right\vert \leq O(\delta_{\eps,l})
\end{equation}
as $\eps\to 0$ and
\begin{equation} \label{eq:bilinearenergyboundary} B\left(\widetilde{\Phi}_\eps,\widetilde{\Phi}_\eps\right) + \int_{R_{l,\eps}^{\pm}} \vert \nabla \widetilde{\Phi}_\eps \vert^2_{\xi} dA_\xi \leq O(\delta_{\eps,l}) \end{equation}
as $\eps \to 0$. 

\medskip

\noindent \textbf{Proof of the Step 1:} We first prove  \eqref{eq:estmeanvalue}. Using that $\vert \widetilde{\Phi}_\eps \vert \leq 1$, we have that
\begin{align*} \left\vert \int_{\mathbb{S}^1} \widetilde{\Phi}_\eps d\theta \right\vert & \leq \left\vert \int_{I_{l,\eps}} \widetilde{\Phi}_\eps dL_{\xi} \right\vert +  \left\vert \int_{J_{l,\eps}} \widetilde{\Phi}_\eps dL_{\xi} \right\vert + \left\vert \beta_\eps(\widetilde{\Phi}_\eps,1)- \int_{\partial\Sigma_{l,\eps}} \widetilde{\Phi}_\eps dL_{g_\eps} \right\vert +  \vert \beta_\eps(\widetilde{\Phi}_\eps-\Phi_\eps,1) \vert \\
& \leq \eps l + 2 \eps + d_{g_\eps}( \beta_\eps , dL_{g_\eps}) \Vert_{g_\eps} \Vert \widetilde{\Phi}_\eps \Vert_{H^1,g_\eps} L_{g_\eps}(\partial\Sigma_{l,\eps} ) + \beta_\eps(\omega_\eps-1,1) . \end{align*}
Now, using  \eqref{eq:clomegacloseto1}, we have
\begin{equation}\label{eq:h1normphieps} \Vert \widetilde{\Phi}_\eps \Vert_{H^1,g_\eps}^2 = \int_{\partial \Sigma_{l,\eps}^\pm} \vert\widetilde{\Phi}_\eps\vert^2 dL_{g_\eps} + \int_{\Sigma_{l,\eps}^\pm}\vert \nabla \widetilde{\Phi}_\eps \vert_{g_\eps}^2 dA_{g_\eps} \leq L_{g_\eps}(\Sigma_{l,\eps}^\pm ) + \sigma_\eps + O(\delta_{l,\eps})\end{equation}
as $\eps \to 0$. Together with the second line in \eqref{eq:ekeland}, and another use of \eqref{eq:clomegacloseto1} with 
$$\beta_\eps(\omega_\eps-1,1) \leq \beta_\eps(\omega_\eps,\omega_\eps) - \beta_\eps(1,1),$$
we obtain \eqref{eq:estmeanvalue}. We now prove \eqref{eq:bilinearenergyboundary}:
\begin{align*} & B\left(\widetilde{\Phi}_\eps,\widetilde{\Phi}_\eps\right) +  \int_{R_{l,\eps}^{\pm}} \vert \nabla \widetilde{\Phi}_\eps \vert^2_{\xi} dA_\xi \leq \int_{\Sigma_{l,\eps}^{\pm}} \left(\vert \nabla \widetilde{\Phi}_\eps \vert^2_{g_\eps} - \vert \nabla \Phi_\eps \vert^2_{g_\eps}\right) dA_{g_\eps} + (\sigma_\eps \beta_\eps(1,1) - 2\pi) \\
&  + \left(2\pi - \beta_\eps(1,1)\right) +   \left(\beta_\eps(\Phi_\eps,\Phi_\eps)- \beta_\eps\left(\widetilde{\Phi}_\eps,\widetilde{\Phi}_\eps\right)\right)  + \left( \beta_\eps\left(\widetilde{\Phi}_\eps,\widetilde{\Phi}_\eps\right)- \int_{\partial \Sigma_{l,\eps}^{\pm}} \left\vert \widetilde{\Phi}_\eps \right\vert^2 dL_{g_\eps} \right) \\ & + \int_{I_{l,\eps}}  \left\vert \widetilde{\Phi}_\eps \right\vert^2 dL_{\xi}  + \frac{1}{2\pi} \left\vert \int_{\mathbb{S}^1} \tilde{\Phi}_\eps d\theta \right\vert^2 = I+II+III+IV+V+VI+VII. \end{align*}
In order to estimate $I$ we will use \eqref{eq:clomegacloseto1} together with the computation:
$$ \vert \nabla \Phi_\eps \vert^2_{g_\eps} = \omega_\eps^2  \vert \nabla \widetilde{\Phi}_\eps \vert^2_{g_\eps} + \vert \widetilde{\Phi}_\eps \vert^2  \vert \nabla \omega_\eps \vert^2_{g_\eps} + \omega_\eps \langle \nabla \omega_\eps \nabla  \vert \widetilde{\Phi}_\eps \vert^2 \rangle_{g_\eps}   $$
that implies using $\tilde{\theta}_\eps^2 = - \vert \widetilde{\Phi}_\eps \vert^2$ and $\omega_\eps \geq 1$:
$$ \vert \nabla \widetilde{\Phi}_\eps \vert^2_{g_\eps} -  \vert \nabla \Phi_\eps \vert^2_{g_\eps} = (1-\omega_\eps^2) \vert \nabla \widetilde{\Phi}_\eps \vert^2_{g_\eps} - \vert \widetilde{\Phi}_\eps \vert^2  \vert \nabla \omega_\eps \vert^2_{g_\eps} + 2 \theta_\eps \langle\nabla \omega_\eps \nabla \frac{\theta_\eps}{\omega_\eps}\rangle_{g_\eps} \leq 2 \frac{\theta_\eps}{\omega_\eps} \langle\nabla \omega_\eps \nabla \theta_\eps \rangle_{g_\eps}$$
to obtain
$$ I \leq 2 \left(\int_{\Sigma_{l,\eps}^\pm} \vert \nabla \omega_\eps \vert_{g_\eps}^2 dA_{g_\eps}\right)^{\frac{1}{2}} \left(\int_{\Sigma_{l,\eps}^\pm} \vert \nabla \theta_\eps \vert_{g_\eps}^2 dA_{g_\eps}\right)^{\frac{1}{2}} = O(\delta_{l,\eps}) $$
as $\eps \to 0$. The assumption $\sigma_\eps \beta_\eps(1,1) \leq 2\pi$ and that $\beta_\eps(1,1)\geq 2\pi$ since $\beta_\eps \in \mathcal{A}$ give $II+III \leq 0$. 
Notice now that $\beta_\eps$ acts as a linear from on squares of $H^1$ functions, we have from $\vert \widetilde{\Phi}_\eps \vert\leq 1$, $\omega_\eps^2 \geq 1$ and \eqref{eq:clomegacloseto1}, that
$$ IV \leq \beta_\eps( \omega_\eps,\omega_\eps ) - \beta_\eps(1,1) \leq O(\delta_{l,\eps}) $$
as $\eps \to 0$.
The second line in \eqref{eq:ekeland}, and \eqref{eq:h1normphieps} imply
$$ V \leq O(\delta_{l,\eps}) $$
as $\eps \to 0$. That $\vert \widetilde{\Phi}_\eps \vert \leq 1$ implies $VI \leq 2l\eps$ and the use of \eqref{eq:estmeanvalue} to estimate $VII$ completes the proof of \eqref{eq:bilinearenergyboundary}.

\medskip

\noindent \textbf{Step 2:}
We now set on $\mathbb{S}^1$
$$ \widetilde{\Phi}_\eps = F_\eps + \int_{\mathbb{S}^1} \widetilde{\Phi}_\eps \frac{d\theta}{2\pi} + R_\eps  $$
where $F_\eps$ is the orthonormal projection of $\widetilde{\Phi}_\eps$ on $Span\{\cos \theta,\sin \theta\}$ with respect to the $L^2$ norm of $\mathbb{S}^1$ and we still denote $R_\eps$ its harmonic extension in $\mathbb{D}$. Then:
\begin{equation} \label{eq:repsto0}  \Vert R_\eps \Vert_{H^1(\mathbb{D})}^2 \leq O(\delta_{\eps,l}) \end{equation}
as $\eps \to 0$ and
\begin{equation} \label{eq:massrepsthin} \left\vert \frac{1}{\eps} \int_{J_{l,\eps}^+} R_\eps \right\vert + \left\vert \frac{1}{\eps} \int_{J_{l,\eps}^-} R_\eps \right\vert \leq O\left( \sqrt{\ln \frac{1}{\eps}}  \Vert R_\eps \Vert_{H^1(\mathbb{D})}\right) \end{equation}
as $\eps \to 0$.

\medskip

\noindent \textbf{Proof of Step 2:} We first prove \eqref{eq:repsto0}.
$$ B(R_\eps,R_\eps) = B(R_\eps,\widetilde{\Phi}_\eps) \leq \sqrt{B(\widetilde{\Phi}_\eps,\widetilde{\Phi}_\eps)B(R_\eps,R_\eps)}. $$
We deduce from \eqref{eq:bilinearenergyboundary} that
$$ B(R_\eps,R_\eps)\leq O(\delta_{l,\eps}) $$
as $\eps \to 0$. We denote by $E_\sigma(\mathbb{D},\xi)$ the set of Steklov eigenfunctions associated to the eigenvalue $\sigma$. Using the latter estimate and that $R_\eps \in \bigoplus_{\sigma \geq 2} E_\sigma(\mathbb{D},\xi)$.
$$ \int_\mathbb{D} \vert \nabla R_\eps\vert^2_{\xi} dA_\xi = B(R_\eps,R_\eps) +   \int_{\mathbb{S}^1}R_\eps^2 d\theta \leq B(R_\eps,R_\eps) + \frac{1}{2} \int_\mathbb{D} \vert \nabla R_\eps\vert^2_{\xi} dA_\xi $$
so that substracting the right-hand term, we estimate the square of the $H^1$ norm of $R_\eps$:
$$\int_{\mathbb{S}^1}R_\eps^2 d\theta + \int_\mathbb{D} \vert \nabla R_\eps\vert^2_{\xi} dA_\xi \leq \frac{3}{2} \int_\mathbb{D} \vert \nabla R_\eps\vert^2_{\xi} dA_\xi \leq 3 B(R_\eps,R_\eps) \leq O(\delta_{l,\eps})$$
as $\eps \to 0$ and this leads to \eqref{eq:repsto0}. 

Now, we prove  \eqref{eq:massrepsthin}. We denote by $\widetilde{R}_\eps : \mathbb{R}^2 \to \R$ the extension of $R_\eps : \mathbb{D} \to \R$ on $\mathbb{R}^2$ via the inversion $i(x) = \frac{x}{\vert x \vert^2}$: $\widetilde{R}_\eps(x) = R_\eps(i(x))$ if $x \in \mathbb{R}^2 \setminus \mathbb{D}$. Then, denoting 
$$f_\eps^\pm(r) = \frac{1}{2\pi r} \int_{0}^{2\pi} \widetilde{R}_\eps(\pm 1 + re^{i\theta})d\theta, $$
we have
\begin{align*}   f_\eps^\pm(1) - f_\eps^\pm(\eps) 
& = \frac{1}{2\pi} \int_{\eps}^{1} \int_0^{2\pi} \partial_r (\ln r)\partial_r  \widetilde{R}_\eps(\pm 1 + re^{i\theta}) rdrd\theta 
\\\ & = \frac{1}{2\pi} \int_{\mathbb{D}(\pm 1) \setminus \mathbb{D}_{\eps}(\pm 1)} \langle\nabla \ln \vert x \vert \nabla \widetilde{R}_\eps \rangle_\xi dA_\xi \\
& \leq \frac{1}{\sqrt{2\pi}} \left(\int_{\eps}^{1} \frac{dr}{r}\right)^{\frac{1}{2}} \left(\int_{\mathbb{R}^2} \vert \nabla \widetilde{R}_\eps \vert_\xi^2 dA_\xi\right)^{\frac{1}{2}}  \leq \frac{ \sqrt{\ln\frac{1}{\eps}} \Vert R_\eps \Vert_{H^1(\mathbb{D})}}{\sqrt{\pi}} .  \end{align*}
By standard trace Poincaré inequalities and definition of $\widetilde{R}_\eps$, we have that
$$ \vert f_\eps^\pm(1) \vert \leq C \Vert \widetilde{R}_\eps \Vert_{H^1(\mathbb{D}_2)} \leq C' \Vert R_\eps \Vert_{H^1(\mathbb{D})} $$
By standard trace Sobolev inequalities again, 
$$ \left\vert \frac{1}{\eps} \int_{J_{l,\eps}^{\pm}} R_\eps \right\vert \leq C \left( f_\eps^{\pm}(\eps) + \left(\int_{\mathbb{D}_{2\eps}(\pm 1)} \vert \nabla \widetilde{R}_\eps \vert^2_\xi dA_\xi \right)^{\frac{1}{2}} \right) $$
and all the previous inequalities lead to \eqref{eq:massrepsthin}.

\medskip

\noindent \textbf{Step 3:} Final argument

\medskip

Since $F_\eps$ lies in a finite dimensional set, up to the extraction of a subsequence, it converges in $\mathcal{C}^2$ to $F : \mathbb{D} \to \mathbb{S}^1$. By  \eqref{eq:estmeanvalue} and \eqref{eq:repsto0}, the weak limit of $\widetilde{\Phi}_\eps$ in $H^1_{loc}(\mathbb{D}\setminus \{-1,1\})$ has to be equal to $F$. Moreover, for any compact subset $I$ of $\mathbb{S}^1 \setminus \{-1,1\}$, we have that $\vert \widetilde{\Phi_\eps} \vert^2 = 1 - \vert \widetilde{\theta}_\eps \vert^2$ in $\mathbb{S}^1$ and $\int_{I} \vert \widetilde{\theta}_\eps \vert^2 d\theta = O(\delta_{\eps,l})$ as $\eps \to 0$ because of the third line in \eqref{eq:eulerlagrange}. Then, $\vert F \vert^2 = 1$ in $\mathbb{S}^1$. We recall that the coordinates of $F$ belong to $Span(\cos\theta,\sin\theta)$. We obtain that there is $\alpha \in \R$ such that for $\theta \in \R$,
\begin{equation}\label{eq:F}F(e^{i\theta}) = (\cos(\theta+\alpha),\sin(\theta+\alpha)).\end{equation}
Now, we compute
\begin{align*} \left\vert \frac{1}{\eps} \int_{J_{l,\eps}^+} \widetilde{\Phi}_\eps d\theta - \frac{1}{\eps} \int_{J_{l,\eps}^-} \widetilde{\Phi}_\eps d\theta \right\vert = \frac{1}{2\eps} \left\vert \int_{C_{l,\eps}} \partial_y \widetilde{\Phi}_\eps \right\vert \leq \frac{( l \eps^2)^{\frac{1}{2}}}{\eps} \left(\int_{C_{l,\eps}} \vert \nabla \widetilde{\Phi}_\eps \vert^2_\xi dA_\xi \right)^{\frac{1}{2}}  \leq O(\left(l\delta_{\eps,l}\right)^{\frac{1}{2}}) \end{align*}
 as $\eps \to 0$, where the latter inequality comes from \eqref{eq:bilinearenergyboundary}. We deduce from the definition of $F_\eps$ and \eqref{eq:estmeanvalue} that
$$ \left\vert \frac{1}{\eps} \int_{J_{l,\eps}^+} F_\eps d\theta - \frac{1}{\eps} \int_{J_{l,\eps}^-} F_\eps d\theta \right\vert \leq O(\left(l\delta_{\eps,l}\right)^{\frac{1}{2}}) + \left\vert \frac{1}{\eps} \int_{J_{l,\eps}^+} R_\eps \right\vert + \left\vert \frac{1}{\eps} \int_{J_{l,\eps}^-} R_\eps \right\vert  $$
as $\eps \to 0$. Now, from \eqref{eq:massrepsthin}, we deduce that
$$ \left\vert \frac{1}{\eps} \int_{J_{l,\eps}^+} F_\eps d\theta - \frac{1}{\eps} \int_{J_{l,\eps}^-} F_\eps d\theta \right\vert \leq O\left(\left(l\delta_{\eps,l}\right)^{\frac{1}{2}} + \sqrt{\ln\frac{1}{\eps}}  \Vert R_\eps \Vert_{H^1(\mathbb{D})} \right). $$
Using \eqref{eq:repsto0} and passing to the limit as $\eps \to 0$, we obtain that
$$ F(-1) = F(1). $$
It contradicts \eqref{eq:F}.

\end{proof}

\begin{remark}
The proof of Proposition \ref{prop:construction>2pi} would be simpler if we directly have maximizers in conformal classes $(\Sigma,[g])$ that satisfy $\sigma_1(\Sigma,[g])= 2\pi$. We use here Ekeland's variational principle because we do not a priori know if there are maximizers in this case. 
\end{remark}

\appendix
\section{
Rigidity result on annuli and Möbius bands 
\\
\small{by Henrik Matthiesen and Romain Petrides}
} \label{sec4}
Before giving a proof of Theorem \ref{thm_annuli}, we recall several results
\begin{theo}[{\cite[Theorem 2]{petrides-3}}] \label{thm_existence}
Let $(\Sigma,g)$ be a compact surface with non-empty boundary such that
$$
\sigma_1(\Sigma,[g]) > 2 \pi,
$$
then there is a smooth positive function $\beta : \Sigma \to \R_+^\star$ such that
\begin{equation} \label{eq_max}
\bar{\sigma}_1(\Sigma, g,\beta)
=
\sigma_1(\Sigma,[g])
\end{equation}
Moreover, for any sequence $(g_k)_{k \in \mathbb{N}}$ of metrics with
\begin{equation} \label{eq_no_bubbling}
\liminf_{k \to \infty}  \sigma_1(\Sigma,[g_k]) > 2 \pi,
\end{equation}
the sequence $(\beta_k)_{k \in \mathbb{N}}$ as in \eqref{eq_max} is smoothly precompact.
\end{theo}
We remark that the last item is not explicitly stated in \cite[Theorem 2]{petrides-3}.
Instead it easily follows from the characterization of these maximizing metrics in terms of free boundary harmonic maps.
Along a sequence enjoying \eqref{eq_no_bubbling} there can not be any bubbling of these harmonic maps.
This is handled in \cite{petrides-3} even under much weaker assumptions.

Next, we state the rigidity results by Fraser--Schoen.
\begin{theo}[{\cite[Theorem 1.2 and Theorem 1.4]{fs}}] \label{thm_fs}
Let $\Phi \colon \Sigma \to \mathbb{B}^N$ be a minimal, free boundary immersion by first Steklov eigenfunctions.
If $\Sigma$ is an annulus, then $\Sigma$ is homothetic to the critical catenoid.
If $\Sigma$ is a M{\"o}bius band, then $\Sigma$ is homothetic to the criticial M{\"o}bius band.
\end{theo}

We also have the following comparison result for Steklov eigenvalues.

\begin{theo}[{\cite{dittmar}, see also \cite[Example 4.2.5.]{gp}}] \label{thm_comp_flat}
For $\eps>0$ sufficiently small, we have 
$$
\bar{\sigma}_1(\mathbb{D} \setminus \mathbb{D}_\eps,\xi) > 2 \pi.
$$
\end{theo}

We need the analogous result for M{\"o}bius bands.
Let $M_\eps$ the M{\"o}bius band obtained as follows.
We glue together two copies $A_1$ and $A_2$ of $\overline{B}_1 \setminus B_\eps$ along $\partial B_\eps$ and identify points by the involution given by
$\iota (x_1)=-x_2$, where $x_1 \in A_1$ and $x_2$ denotes the point with the same coordinates as $x_1$ but in $A_2$.
Note that the metric on $M_\eps$ is only Lipschitz, but it can easily be approximated by a sequence of smooth metrics such that the length of the boundary and the first Steklov eigenvalue converge.

\begin{prop} 
We have that
$$
\bar{\sigma}_1(M_\eps,\xi) > 2\pi
$$
for $\eps>0$ sufficiently small.
\end{prop}

The argument is analogous to (and in fact easier than) the proof of Theorem \ref{thm_comp_flat}. 
We record it below for the convenience of the reader.

\begin{proof}
Note that there is a second isometric involution on $A_1 \cup A_2$ given by $\tau(x_1)=x_2$ in the notation above.
Since $\tau$ is an isometric involution it acts on any eigenspace of the Dirichlet-to-Neumann operator and splits these into $\pm 1$-eigenspaces.
Note that the $+1$ eigenspace corresponds to eigenfunctions on $A_\eps=B_1 \setminus B_\eps$ of the Steklov problem with Neumann boundary conditions along $\partial B_\eps$.
Analogously the $-1$ eigenspace corresponds to eigenfunctions of the Steklov problem on $A_\eps$ with Dirichlet boundary conditions along $\partial B_\eps$.
These eigenvalues can be computed explicitly as follows.
Write $u_k=(A_kr^k+A_{-k}r^{-k}) T(k \theta)$
for $T$ either $\cos$ or $\sin$.
Then it easily checked that for $u_k$ to be a
Dirichlet eigenfunction we need to have that $A_k = -A_{-k}\eps^{-2k}$ which leads to $\sigma = k \frac{\eps^{-2k}+1}{\eps^{-2k}-1}>k$.
Additionally, there is an eigenfunction given by $\log(r/\eps)$, which is not $\iota$ invariant.
Analogously, $u_k$ is a Neumann eigenfunction if and only if $A_k=A_{-k}\eps^{-2k}$.
The corresponding eigenvalue is given by $k \frac{\eps^{-2k}-1}{\eps^{-2k}+1} \to k$ as $\eps \to 0$.
However, for $k=1$ none of these eigenfunctions is $\iota$ invariant.
\end{proof}

We also need that a similar results holds near the other end of the moduli space. 
This result is much more subtly proved thanks to Proposition \ref{prop:construction>2pi}.
\begin{prop} \label{theo:strictineq} There is $r_k \to 1$ such that
\begin{equation} \label{ineqstrict1}
\bar{\sigma}_1(\mathbb{D}\setminus \mathbb{D}_{r_k},\xi) >2\pi
\end{equation}
as $k\to +\infty$ and there is $r_k' \to 1$ such that
\begin{equation} \label{ineqstrict2}
\bar{\sigma}_1(M_{r_k'},\xi) > 2\pi
\end{equation}
as $k\to +\infty$
\end{prop}

This proposition will be a direct consequence of Proposition \ref{prop:construction>2pi} and the following lemma:
\begin{lem} \label{lem_conformal}
Let $\phi \colon [0,r] \times [0,1] \to [0,R] \times [0,1]$ be conformal embedding, smooth in the interior, such that $\phi([0,r] \times \{i\}) \subset [0,R] \times \{i\}$, for $i=0,1$.
Then we have that $r \leq R$.
\end{lem}

\begin{proof}
Since $\phi$ is conformal, we have that 
$$
\phi^\star (dx^2+dy^2) = e^{2 \omega} (dx^2 + dy^2)
$$
for a function $\omega$ that is smooth in the interior.
By assumption, we have that
$$
1 
\leq \left( L (\phi(\{x\} \times [0,1])) \right)^2 
= \left(\int_0^1 e^{\omega(x,y)} dy \right)^2
\leq \int_0^1 e^{2 \omega(x,y)} dy,
$$
where we have used Jensen's inequality and note that this remains valid also if $\phi(\{x\} \times [0,1])$ is not a rectifiable curve.
This implies that
$$
r 
= \int_0^r  dx
\leq \int_0^r \int_0^1 e^{2 \omega(x,y)} dy dx
=Area(\phi) \leq R
$$
\end{proof}

\begin{lem} \label{lem_conformal_2}
We recall that $\Sigma_{l,\eps}^+$ is an annulus and
let $\Phi \colon \overline{B}_1 \setminus B_{r} \to \Sigma_{l,\eps}^+$ be a conformal homeomorphism, which is smooth in the interior.
Then we have that $r \geq \exp(\frac{-2\pi}{l})$.
\end{lem}

\begin{proof}

We conformally parametrize the universal covering $\widetilde \Sigma_{l,\eps}^+$ of $\Sigma_{l,\eps}^+$ by
$(-\infty,\infty) \times [0,1]$ with deck transformations generated by $(x,y) \mapsto (x+R,y)$ for $R>0$, which uniquely determines $R$.
We have a conformal embedding
$$
\phi \colon [0,l] \times [0,1] \to \Sigma_{l,\eps}^+ 
$$
given explicitly by
$$
\phi(x,y)=\left(\eps x - \frac{l \eps}{2},\eps y - \frac{\eps}{2} \right) \in R_{l,\eps} \subseteq \Sigma_{l,\eps}^+.
$$
We then lift $\phi$ to a map
$$
\tilde \phi \colon [0,l] \times [0,1] \to  [0,R] \times [0,1] \subset \widetilde \Sigma_{l,\eps}^+
$$
with
$
\tilde \phi([0,l] \times \{i\}) \subset  [0,R] \times \{i\}.
$
Lemma \ref{lem_conformal} applied to $\tilde \phi$ gives that 
$$
l \leq R,
$$
from which the claim easily follows recalling that $R = \frac{2\pi} {\log(1/r)}$.
\end{proof}

\begin{remark} 
Notice that in the proof of Theorem \ref{thm_annuli} we just need a sequence $r_k \to 1$ such that $\bar{\sigma}_1(\mathbb{D}\setminus \mathbb{D}_{r_k}) > 2\pi$ (and the same for möbius band). It is likely that we should directly deduce from Proposition \ref{prop:construction>2pi} that $\bar{\sigma}_1(\mathbb{D}\setminus \mathbb{D}_{r}) > 2\pi$ for any $r$ close to $1$.
\end{remark}

\begin{proof}[Proof of Theorem \ref{thm_annuli}]
We give the proof for $\Sigma$ an annulus, the case of M{\"o}bius bands is completely analogous. We write $A_r = \overline{B_1} \setminus B_r \subset \R^2$ and note that any compact annulus is conformal to some $A_r$ for a unique $r \in (0,1)$.

Consider the functional 
$\tilde{\sigma}_1 : (0,1) \to (0,\infty)$ given by
$$
\tilde{\sigma}_1(r) = \sigma_1(A_r, [\xi]),
$$
where $\xi$ denotes the canonical flat metric on $A_r$.
We want to show that $\tilde{ \sigma}_1(r)>2\pi$ for any $r \in (0,1)$.
From Theorem \ref{theo:largeineq}, we know that
$$ \tilde{\sigma}_1(r) \geq 2\pi. $$
Moreover, we also know that
\begin{equation} \label{eq_limit}
\lim_{r \to 0} \tilde{\sigma}_1(r) = \lim_{r \to 1} \tilde{\sigma}_1(r) = 2\pi,
\end{equation}
see \cite[Proposition 4.4]{fs}.

Let $r_\star \in (0,1)$ be chosen such that $A_{r_\star}$ is conformal to the critical catenoid.
Suppose now towards a contradiction that there is some $s \in (0,1)$ such that $\overline{\sigma}_1(s) \leq 2 \pi$.
We assume that $s > r_\star $.
The argument for the other case is identical because of \eqref{eq_limit} and we use Theorem \ref{thm_comp_flat} instead of \eqref{ineqstrict1}.
Then, by \eqref{ineqstrict1}, $\sup_{r \in [s,1)} \tilde{\sigma}_1(r) > 2 \pi$. We now claim that there is $ t \in (s,1)$ such that
$$
 \tilde{\sigma}_1(t) = \max_{r \in [s,1)} \overline{\sigma}_1(r) > 2 \pi.
$$
Indeed, we set $\Sigma = A_{\frac{1}{2}}$ and we let $\theta_r : \Sigma \to A_r$ be the family of diffeomorphisms 
$$\theta_r(x) = \frac{\vert x \vert + 1- 2r}{2(1-r)} \frac{x}{\vert x \vert}.$$
We can maximize the functional $ (r,\beta) \mapsto \overline{\sigma}_1(\Sigma,\theta_r^{\star}(\xi),\beta) $ on the set
$$ \left\{ r \in [s,1) ; \sigma_1(A_r,[\xi]) \geq \frac{\sup_{r \in [s,1)} \tilde{\sigma}_1(r) + 2 \pi}{2}  \right\} \times \mathcal{C}_{>0}^\infty(\partial \Sigma).$$
Indeed, for a maximizing sequence $(r_k,\beta_k)_{k\in \mathbb{N}}$, we have that for $k$ large enough $r_k \leq 1-c$ for a positive constant $c$ because of \eqref{eq_limit}. Then up to the extraction of a subsequence, $r_k \to r$ as $k\to+\infty$ and $\beta_k \to \beta$ as $k \to +\infty$ thanks to Theorem \ref{thm_existence}.

It then follows that $(A_r, \xi, \beta \circ \theta_r^{-1})$ is a local maximum of the normalized first Steklov eigenvalue and hence induced by a branched, free-boundary minimal immersion into $\mathbb{B}^N$ by first eigenfunctions thanks to Theorem \ref{theo:critcombS}.
But by the uniqueness of the critical catenoid, see Theorem \ref{thm_fs} above, this is impossible for $t \neq r_\star$.
\end{proof}

\bibliographystyle{alpha}
\end{document}